\documentclass[11pt]{amsart}
\usepackage{amscd,amssymb}
\usepackage{amsthm,amsmath,amssymb}
\usepackage[matrix,arrow]{xy}

\theoremstyle{definition}
\newtheorem{example}[equation]{Example}
\newtheorem{definition}[equation]{Definition}
\newtheorem{theorem}[equation]{Theorem}
\newtheorem{lemma}[equation]{Lemma}
\newtheorem{corollary}[equation]{Corollary}

\newtheorem*{question*}{Question}

\newtheorem*{problem*}{Problem}

\theoremstyle{remark}
\newtheorem{remark}[equation]{Remark}

\makeatletter\@addtoreset{equation}{section} \makeatother

\def\Q {\mathbb{Q}}
\def\C {\mathbb{C}}
\def\P {\mathbb{P}}
\def\Z {\mathbb{Z}}
\def\F {\mathbb{F}}

\def\LL {\mathcal{L}}
\def\PP {\mathcal{P}}
\def\O {\mathcal{O}}
\def\CC {\mathcal{C}}

\def\A {\mathrm{A}}
\def\SS {\mathrm{S}}
\def\PSL {\mathrm{PSL}}
\def\SL {\mathrm{SL}}
\def\GL {\mathrm{GL}}
\def\PGL {\mathrm{PGL}}
\def\Aut {\mathrm{Aut}}
\def\Pic {\mathrm{Pic}}
\def\Sym {\mathrm{Sym}}

\def\mult {\mathrm{mult}}

\def\Sing {\mathrm{Sing}}
\def\rk {\mathrm{rk}}

\def\PSp {\mathrm{PSp}}
\def\SU {\mathrm{SU}}
\def\PSU {\mathrm{PSU}}
\def\HaJ {\mathrm{HaJ}}

\def\le {\leqslant}
\def\ge {\geqslant}

\author{Ivan Cheltsov  and Constantin Shramov}

\title{Weakly-exceptional singularities in higher dimensions}%

\address{University of Edinburgh, Edinburgh EH9 3JZ, UK, \texttt{I.Cheltsov@ed.ac.uk}}

\address{Steklov Institute of Mathematics, Moscow 119991, Russia, \texttt{shramov@mccme.ru}}

\address{Laboratory of Algebraic Geometry, GU-HSE, 7 Vavilova Street, Moscow, 117312, Russia}

\thanks{Unless explicitly stated otherwise, varieties are assumed to be
projective,~normal~and~complex. Throughout the paper we use the
standard language of the singularities of pairs (see \cite{Ko97}).
By strictly log canonical singularities we mean log canonical
singularities that are not Kawamata log terminal (see
\cite[Definition~3.5]{Ko97}).}

\begin{document}

\begin{abstract}
We show that infinitely many Gorenstein weakly-exceptional quotient
singularities exist in all dimensions, we prove
a~weak-exceptionality criterion for five-dimensional quotient
singularities, and we find a sufficient condition for being
weakly-exceptional for six-dimensional quotient singularities. The
proof is naturally linked to various classical geometrical
constructions related to subvarieties of small degree in
projective spaces, in particular Bordiga surfaces and Bordiga
threefolds.
\end{abstract}

\maketitle

\section{Introduction}
\label{section:intro}

Linear representations of finite groups induce their action on
polynomial functions. Invariant theory studies polynomial
functions that are invariant under the transformations from a
given linear group. These functions form a ring, which is called a
ring of invariants. From the point of view of Algebraic Geometry,
the rings of invariants are algebraic counterparts of the
quotients of the vector spaces by these groups, which are usually
singular spaces and are called quotient singularities.

Finite subgroups in $\SL_2(\C)$ have been classified long time
ago. The corresponding quotients by these groups are famous
$\mathbb{A}$-$\mathbb{D}$-$\mathbb{E}$ singularities, which are
also known by other names (Kleinian singularities, Du Val
singularities, rational surface double points, two-dimensional
canonical singularities etc). Taking into account the classical
double cover $\mathrm{SU}_2(\C)\to\mathrm{SO}_3(\mathbb{R})$ and
basic representation theory, we see that the singularities of type
$\mathbb{A}$ correspond to plane rotations, the singularities of
type $\mathbb{D}$ correspond to the groups of symmetries of
regular polygons, and the singularities of type $\mathbb{E}$
corresponds to the groups of symmetries of Platonic solids.

Shokurov suggested a higher dimensional generalization of the
singularities of type $\mathbb{E}$ and of both types $\mathbb{D}$
and $\mathbb{E}$. He called them exceptional and
weakly-exceptional, respectively. It turned out that exceptional
and weakly-exceptional singularities are related to the Calabi
problem for orbifolds with positive first Chern class (see
\cite{ChPaSh10}, \cite{ChSh09}). Exceptional quotient
singularities of dimensions $3$, $4$, $5$ and $6$ have been
completely classified by Markushevich, Prokhorov, and the authors
(see \cite{MarPr99}, \cite{ChSh09}, \cite{ChSh10}). Moreover, we
proved in \cite{ChSh10} that seven-dimensional exceptional
quotient singularities are ``too exceptional'' --- they simply do
not exist. On the other hand, weakly-exceptional quotient
singularities have been less popular despite the fact that their
definition is much simpler than the definition of exceptional
ones. In this paper, we will try to fill this gap.

What is special about singularities of types $\mathbb{D}$ and
$\mathbb{E}$ and how to generalize them to higher dimensions?
There are many possible answers to these questions. One of them
involves the dual graphs of their minimal resolutions of
singularities. Namely, the dual graphs of the minimal resolution
of singularities of any two-dimensional quotient singularity of
type $\mathbb{D}$ and $\mathbb{E}$ always have a ``fork'', i.e. a
special curve that intersects three other exceptional curves in
the graph. The singularities of type $\mathbb{A}$ lack this
property. Surprisingly, this property of having a  ``very
special'' exceptional divisor on some resolution of singularities
can be generalized to higher dimensions for quotient singularities
and other ``mild'' singularities.

Let $(V\ni O)$ be a~germ of a~Kawamata log terminal singularity
(see \cite[Definition~3.5]{Ko97}). Then there exists (see, for
example, \cite[Theorem~3.6]{ChSh09}) a~birational morphism
$\pi\colon W\to V$ whose exceptional locus consists of a single
irreducible divisor $E\subset W$ such that $O\in\pi(E)$, the~log
pair $(W,E)$ has purely log terminal singularities (see
\cite[Definition~3.5]{Ko97}), and $-E$ is a~$\pi$-ample
$\mathbb{Q}$-Cartier divisor. The birational morphism $\pi\colon
W\to V$ is said to be a~\emph{plt blow up} of the~singularity
$(V\ni O)$.

\begin{example}
\label{example:intro-ADE} Suppose that $(V\ni O)$ is a
two-dimensional Du Val singularity. If it is of type~$\mathbb{A}$,
then let us choose $\pi\colon W\to V$ to be any partial resolution
of singularities that contracts exactly one curve to the point
$O$. In the case when  $(V\ni O)$  is a singularity of type
$\mathbb{D}$ or $\mathbb{E}$, let us choose $\pi\colon W\to V$ to
be the partial resolution of singularities contracting exactly one
curve that corresponds to the ``central'' vertex of the dual graph
of the minimal resolution of singularities of $(V\ni O)$ to the
point $O$. Then $\pi$ is a plt blow up of the singularity $(V\ni
O)$ (cf. Example~\ref{example:intro-ADE-weakly-exceptional}).
\end{example}

\begin{example}
\label{example:intro-quasi-homogeneous} Suppose that $(V\ni O)$ is
an isolated quasihomogeneous hypersurface singularity in
$\C^{n+1}$ with respect to some positive integral weights
$(a_{0},\ldots,a_{n})$ such that
$\mathrm{gcd}(a_{0},\ldots,a_{n})=1$. Then $(V,O)$ is given by
$$
\phi\big(x_{0},\ldots,x_{n}\big)=0\subset\mathbb{C}^{n+1}\cong\mathrm{Spec}\Big(\mathbb{C}\big[x_{0},\ldots,x_{n}\big]\Big)
$$
for some quasihomogeneous polynomial
$\phi\in\mathbb{C}[x_{0},\ldots,x_{n}]$ of degree $d$ with respect
to the weights
\mbox{$\mathrm{wt}(x_{1})=a_{1},\ldots,\mathrm{wt}(x_{n})=a_{n}$}.
It is well-known that $(V\ni O)$ is Kawamata log terminal if and
only if $\sum_{i=0}^{n}a_{i}>d$. If this is the case, the weighted
blow up of $\C^{n+1}$ with weights $(a_{0},\ldots,a_{n})$ induces
a plt blow up $\pi$ of $(V\ni O)$. If $n=1$ and $(V\ni O)$ is of
type $\mathbb{D}$ or $\mathbb{E}$, then the choice of weights is
unique, and the morphism $\pi$ constructed in this example
coincides with the morphism~$\pi$ constructed in
Example~\ref{example:intro-ADE}.
\end{example}

\begin{example}
\label{example:intro-quotient} Suppose $(V\ni O)$ is a~quotient
singularity $\mathbb{C}^{n+1}\slash G$, where $n\ge 1$ and $G$ is
a~finite~subgroup in $\GL_{n+1}(\mathbb{C})$. Note that quotient
singularities are always Kawamata log terminal (see
\cite[Remark~0.2.17]{KMM}). Let $\eta\colon\mathbb{C}^{n+1}\to V$
be the~quotient map. Then there is a~commutative diagram
$$
\xymatrix{
U\ar@{->}[d]_{\gamma}\ar@{->}[rr]^{\omega}&&W\ar@{->}[d]^{\pi}\\
\mathbb{C}^{n+1}\ar@{->}[rr]_{\eta}&&V,}
$$
where $\gamma$ is the~blow up of $O$, the~morphism $\omega$ is
the~quotient map that is induced by the~lifted action~of $G$ on
the~variety $U$, and $\pi$ is a~birational morphism. One can
easily check that $\pi$ is a~plt blow~up of
the~singularity~$\C^{n+1}\slash G$. Note that $\pi$ may not
``improve'' singularities of~$V$. For example, if $n=1$,
$G\subset\SL_2(\C)$ and $G\cong\Z_3$, then this construction does
not give a partial resolution of the singularity $(V\ni O)$.
However, if $n=1$, $G\subset\SL_2(\C)$ and $(V\ni O)$ is a
singularity of type $\mathbb{D}$ or $\mathbb{E}$, then the
morphism $\pi$ constructed in this example coincides with the
morphism $\pi$ constructed in Examples~\ref{example:intro-ADE}
and~\ref{example:intro-quasi-homogeneous}.
\end{example}

Keeping in mind Examples~\ref{example:intro-ADE},
\ref{example:intro-quasi-homogeneous}, and
\ref{example:intro-quotient}, one gives

\begin{definition}[{\cite[Definition~4.1]{Pr98plt}}]
\label{definition:weakly-exceptional} We say that $(V\ni O)$ is
\emph{weakly-exceptional} if it has a unique plt blow up.
\end{definition}

Note that weakly-exceptional Kawamata log terminal singularities
do exist (see Example~\ref{example:intro-ADE-weakly-exceptional}).
How to decide whether the singularity $(V\ni O)$ is
weakly-exceptional or not? Surprisingly, the answer to this
question depends only on the log pair $(W,E)$. For instance, a
necessary condition for $(V\ni O)$ to be weakly-exceptional says
that $\pi(E)=O$ (see \cite[Corollary~1.7]{Kud01}). However, it
follows from Example~\ref{example:intro-ADE} that this condition
is very far from being a criterion. To give a criterion for $(V\ni
O)$ to be weakly-exceptional, we must equip $E$ with an extra
structure. Namely, let $R_{1},\ldots,R_{s}$ be all irreducible
components of the locus $\mathrm{Sing}(W)$ of dimension
$\mathrm{dim}(W)-2$ that are contained in $E$. Put
$$
\mathrm{Diff}_{E}\big(0\big)=\sum_{i=1}^{s}\frac{m_{i}-1}{m_{i}}R_{i},
$$
where $m_{i}$ is the~smallest positive integer such that $m_{i}E$
is~Cartier at a~general point of $R_{i}$. The divisor
$\mathrm{Diff}_{E}(0)$ is usually called a \emph{different} (it
was introduced by Shokurov in \cite{Sho93}). It follows from
\cite[Theorem~7.5]{Ko97} that $E$ is a normal variety that has at
most rational singularities, and the log pair
$(E,\mathrm{Diff}_{E}(0))$ is Kawamata log terminal. Thus, if
\mbox{$\pi(E)=O$}, then the~log pair $(E,\mathrm{Diff}_{E}(0))$ is
a~log Fano variety, because $-E$ is $\pi$-ample.

\begin{theorem}[{\cite[Theorem~4.3]{Pr98plt}, \cite[Theorem~2.1]{Kud01}}]
\label{theorem:weakly-exceptional-criterion} The~singularity
$(V\ni O)$ is weakly-exceptional if~and~only~if $\pi(E)=O$ and the
log pair $(E,\mathrm{Diff}_{E}(0)+D_{E})$ is log canonical for
every effective $\mathbb{Q}$-divisor $D_{E}$ on the~variety $E$
such that $D_{E}\sim_{\mathbb{Q}}-(K_{E}+\mathrm{Diff}_{E}(0))$.
\end{theorem}

Let us translate the assertion of
Theorem~\ref{theorem:weakly-exceptional-criterion} into a slightly
different language that uses a global log canonical threshold, that
is an~algebraic counterpart of the so-called $\alpha$-invariant of
Tian introduced in \cite{Ti87}.

\begin{remark}
\label{remark:intro-G-threshold} For a log Fano variety
$(X,B_{X})$ with at most Kawamata log terminal singularities and a
finite group $\bar{G}\subset\mathrm{Aut}(X)$ such that $B_{X}$ is
$\bar{G}$-invariant, the number
$$
\mathrm{sup}\left\{\lambda\in\mathbb{Q} \left|\ %
\aligned &\big(X,B_{X}+\lambda D_{X}\big)\ \text{has Kawamata log terminal singularities}\\
&\text{for every~$\bar{G}$-invariant $\mathbb{Q}$-Cartier effective $\mathbb{Q}$-divisor $D_{X}$}\\
&\text{on the~variety $X$ such that}\ D_{X}\sim_{\mathbb{Q}} -\Big(K_{X}+B_{X}\Big)\\
\endaligned\right.\right\}.%
$$
is denoted by $\mathrm{lct}(X,B_{X},\bar{G})$ and is called the
\emph{global $\bar{G}$-invariant log canonical threshold} of
the~log Fano variety $(X,B_{X})$ (see
\cite[Definition~3.1]{ChSh09}). For simplicity, we put
$\mathrm{lct}(X,\bar{G})=\mathrm{lct}(X,B_{X},\bar{G})$ if
$B_{X}=0$, we put
$\mathrm{lct}(X,B_{X})=\mathrm{lct}(X,B_{X},\bar{G})$ if $\bar{G}$
is trivial, and we put
$\mathrm{lct}(X)=\mathrm{lct}(X,B_{X},\bar{G})$ if $B_{X}=0$ and
$\bar{G}$ is trivial. If $X$ is smooth and $B_{X}=0$, then it
follows from \cite[Theorem~A.3]{ChSh08c} that
$\mathrm{lct}(X,\bar{G})=\alpha_{\bar{G}}(X)$, where
$\alpha_{\bar{G}}(X)$ is the $\alpha$-invariant of Tian of the
Fano variety $X$. If $X$ has at most quotient singularities and
$B_{X}=0$, then it follows from \cite{Ti87}, \cite{Na90} and
\cite{DeKo01} that $X$ admits a $\bar{G}$-invariant orbifold
K\"ahler--Einstein metric if
$\mathrm{lct}(X,G)>\mathrm{dim}(X)\slash(\mathrm{dim}(X)+1)$.
\end{remark}

It follows from Theorem~\ref{theorem:weakly-exceptional-criterion}
that $(V\ni O)$ is weakly-exceptional if~and~only~if $\pi(E)=O$
and $\mathrm{lct}(E,\mathrm{Diff}_{E}(0))\geqslant 1$. It should
be pointed out that
Theorem~\ref{theorem:weakly-exceptional-criterion} is an
applicable criterion. For instance, it can be used to construct
weakly-exceptional singularities of any dimension (cf.
\cite[Example~3.13]{ChSh09}).

\begin{example}
\label{example:intro-ADE-weakly-exceptional} In the notation and
assumptions of Example~\ref{example:intro-ADE}, let $E$ be the
exceptional curve of the plt blow up $\pi\colon W\to V$. If $(V\ni
O)$ is a singularity of type $\mathbb{A}$, then one can easily see
that $\mathrm{lct}(E,\mathrm{Diff}_{E}(0))<1$, which implies that
$(V\ni O)$ is not weakly-exceptional by
Theorem~\ref{theorem:weakly-exceptional-criterion}. Suppose that
$(V\ni O)$ is a singularity of type $\mathbb{D}$ or  $\mathbb{E}$.
Then $W$ is singular along $E$, and $E$ contains exactly three
singular points of the surface $W$. Let us denote these points  by
$P_{1},P_{2},P_{3}$. Then each $P_{i}$ is a singular point of type
$\mathbb{A}_{n_{i}}$ for some non-negative integer $n_{i}$.
Without loss of generality, we may assume that $1\leqslant
n_{1}\leqslant n_2\leqslant n_{3}$. Then
$\mathrm{Diff}_{E}(0)=\sum_{i=1}^{3}\frac{n_{i}}{n_{i}+1}P_{i}$
(see~\cite[Proposition~16.6]{Corti92}).
The log pair $(E,\mathrm{Diff}_{E}(0))$ is Kawamata log terminal,
since $n_{i}/(n_{i}+1)<1$ for every $i\in\{1,2,3\}$. Moreover, we
have $\sum_{i=1}^{3}\frac{n_{i}}{n_{i}+1}<2$, which implies that
$(E,\mathrm{Diff}_{E}(0))$ is a log Fano variety. Then
$$
\mathrm{lct}\Big(E,\mathrm{Diff}_{E}\big(0\big)\Big)=\frac{1-\frac{n_{3}}{n_{3}+1}}{2-\sum_{i=1}^{3}\frac{n_{i}}{n_{i}+1}}=
\left\{%
\aligned
&6 \mbox{ if $(V\ni O)$ is a singularity of type $\mathbb{E}_{8}$,}\\%
&3 \mbox{ if $(V\ni O)$ is a singularity of type $\mathbb{E}_{7}$,}\\%
&2 \mbox{ if $(V\ni O)$ is a singularity of type $\mathbb{E}_{6}$,}\\%
&1 \mbox{ if $(V\ni O)$ is a singularity of type $\mathbb{D}$,}\\%
\endaligned\right.%
$$
which implies that $(V\ni O)$ is weakly-exceptional.\footnote{
Let $(V\ni O)$ be a germ of a Kawamata log terminal surface
singularity.
Then either $(V\ni O)$ is smooth, or the dual graph of
its minimal resolution of singularities is of $\mathbb{A}$,
$\mathbb{D}$ or $\mathbb{E}$ type (see~\cite{Brieskorn68}
and~\cite{Iliev86} for details).
Arguing as in the Du Val case we see that the singularity $(V\ni O)$
is weakly-exceptional if and only if the dual graph
is of $\mathbb{D}$ or $\mathbb{E}$ type.

Moreover, it follows from~\cite[Corollary~1.9]{Kawamata84} that
$(V\ni O)$ is a quotient singularity~$\mathbb{C}^2/G$ for some finite
group $G\subset\mathrm{GL}_2(\mathbb{C})$. However, it is not easy
to give a nice description of such $G$ that correspond to
weakly-exceptional singularities $(V\ni O)$
(cf.~\cite[Example~1.9]{ChSh10}).}
\end{example}

As it clearly follows from
Example~\ref{example:intro-ADE-weakly-exceptional},
weakly-exceptional singularities are a natural generalization of
Du Val singularities of type $\mathbb{D}_{n}$, $\mathbb{E}_{6}$,
$\mathbb{E}_{7}$ and $\mathbb{E}_{8}$. On the other hand,
weakly-exceptional singularities are not classified even in
dimension three. In fact, weakly-exceptional singularities are
classified just in two cases. Firstly, three-dimensional
weakly-exceptional Kawamata log terminal isolated quasihomogeneous
well-formed hypersurface singularities (see \cite{ChPaSh10}),
which provides a lot of examples of K\"ahler--Einstein
two-dimensional Fano orbifolds that are hypersurfaces in weighted
projective spaces (see
Example~\ref{example:intro-quasi-homogeneous} and
Remark~\ref{remark:intro-G-threshold}). Secondly,
three-dimensional and four-dimensional weakly-exceptional quotient
singularities are classified in \cite{MarPr99}, \cite{ChSh09} and
\cite{Sak10} (cf. Theorems~\ref{theorem:weakly-exceptional-dim-3}
and \ref{theorem:Vanya-Kostya-SL-4-weakly-exceptional}). The goal
of this paper, as it follows from its title, is to study
weakly-exceptional quotient singularities in higher dimensions.

Let $G$ be a~finite subgroup in $\GL_{n+1}(\mathbb{C})$, where
$n\geqslant 1$. Let
\mbox{$\phi\colon\GL_{n+1}(\mathbb{C})\to\PGL_{n+1}(\mathbb{C})$}
be the~natural projection. Put $\bar{G}=\phi(G)$ and let us
identify the group $\PGL_{n+1}(\mathbb{C})$ with
$\mathrm{Aut}(\mathbb{P}^{n})$. Recall that an element $g\in G$ is
called a \emph{reflection} (or sometimes a
\emph{quasi-reflection}) if there is a hyperplane in
$\mathbb{P}^{n}$ that is pointwise fixed by $\phi(g)$. To study
the weak-exceptionality of the singularity $\C^{n+1}/G$ one can
always assume that the group $G$ does not contain reflections (cf.
\cite[Remark~1.16]{ChSh10}). In this case we have the following

\begin{theorem}[{\cite[Theorem~3.16]{ChSh09}}]
\label{theorem:weakly-exceptional-quotient} Let $G$ be a finite
subgroup in $\GL_{n+1}(\C)$ that does not contain reflections.
Then the~singularity $\C^{n+1}\slash G$ is weakly-exceptional
$\iff$ $\mathrm{lct}(\mathbb{P}^{n},\bar{G})\geqslant 1$.
\end{theorem}

It should be pointed out that the assumption that $G$ contains no
reflections is crucial for
Theorem~\ref{theorem:weakly-exceptional-quotient} and can not be
removed (see \cite[Example~1.18]{ChSh09}). Recall that a
\emph{semi-invariant} of the group $G$ is a polynomial whose
zeroes define a $\bar{G}$-invariant hypersurface in
$\mathbb{P}^{n}$. If the~group~$G$ has a~semi-invariant of
degree~$d$, then $\mathrm{lct}(\mathbb{P}^{n},\bar{G})\leqslant
d/(n+1)$. Thus, if $G$ does not contain reflections and~$G$ has
a~semi-invariant of degree~at~most~$n$, then $\C^{n+1}/G$ is not
weakly-exceptional by
Theorem~\ref{theorem:weakly-exceptional-quotient}. Moreover, it
follows from Theorem~\ref{theorem:weakly-exceptional-quotient} and
Example~\ref{example:intro-ADE-weakly-exceptional} that if $G$ is
a finite subgroup in $\GL_2(\C)$ that does not contain
reflections, then the~singularity $\C^2/G$ is exceptional if and
only if $G$ has no semi-invariants of degree $1$. A similar result
holds in dimension $3$ due to

\begin{theorem}[{\cite[Theorem~3.23]{ChSh09}}]
\label{theorem:weakly-exceptional-dim-3} Let $G$ be a finite group
in $\GL_3(\C)$ that does not contain reflections. Then
the~singularity $\C^3/G$ is weakly-exceptional if and only if $G$
does not have semi-invariants of degree at most $2$.
\end{theorem}

Representation theory provides various obstructions to weak
exceptionality. For example, it follows
from~\cite[Proposition~2.1]{Pr00} that the subgroup
$G\subset\GL_{n+1}(\C)$ is transitive (i.\,e. the corresponding
$(n+1)$-dimensional representation is irreducible) provided that
the singularity \mbox{$\C^{n+1}\slash G$} is weakly exceptional.
Nevertheless,  we do not expect that
Definition~\ref{definition:weakly-exceptional} can be translated
into the representation-theoretic language in higher dimensions
even in a very restrictive case of quotient singularities. Namely,
$\bar{G}$-invariant subvarieties of large codimension may also
provide obstructions to weak exceptionality as follows from

\begin{theorem}[{\cite[Theorem~4.3]{ChSh09}}]
\label{theorem:Vanya-Kostya-SL-4-weakly-exceptional} Let $G$ be a
finite group in $\GL_4(\C)$ that does not contain reflections.
Then the~singularity $\C^4/G$ is weakly-exceptional if and only if
the~following three conditions are satisfied: the~group $G$ is
transitive, the~group $G$ does not have semi-invariants of degree
at most $3$, and there is no $\bar{G}$-invariant smooth rational
cubic curve in $\mathbb{P}^{3}$.
\end{theorem}

Futhermore, we can construct a non-weakly exceptional
six-dimensional quotient singularity arising from a transitive
finite subgroup without reflections in $\GL_6(\C)$ that has no
semi-invariants of degree at most~$6$ (see
\cite[Example~3.20]{ChSh09} and \cite[Lemma~3.21]{ChSh09}). This,
together with Theorems~\ref{theorem:weakly-exceptional-dim-3} and
\ref{theorem:Vanya-Kostya-SL-4-weakly-exceptional}, shows that the
weak-exceptionality of a quotient
singularity~\mbox{$\C^{n+1}\slash G$} crucially depends on the
absence of certain special $\bar{G}$-invariant subvarieties in
$\mathbb{P}^{n}$. Note that the transitivity of the subgroup $G$
is equivalent to the absence of $\bar{G}$-invariant proper linear
subspaces in~$\P^n$. In particular, we should not expect that
there exists a simple looking criterion for a quotient singularity
$\C^{n+1}\slash G$ to be weakly-exceptional that works in all
dimensions. On the other hand, it is quite natural to expect that
there exists such a sufficient condition (not a criterion). This
is indeed the case.

\begin{definition}
\label{definition:Fano-type} An irreducible normal variety $V$ is
said to be of Fano type if there exists an effective
$\mathbb{Q}$-divisor $\Delta_{V}$ on the variety $V$ such that
$-(K_{V}+\Delta_{V})$ is a~$\mathbb{Q}$-Cartier ample divisor, and
the log pair $(V,\Delta_{V})$ has at most Kawamata log terminal
singularities (see \cite[Definition~3.5]{Ko97}).
\end{definition}

In this paper (see Section~\ref{section:criterion}), we prove the
following

\begin{theorem}
\label{theorem:weakly-exceptional-quotient-criterion} Let $G$ be a
finite group in $\GL_{n+1}(\C)$ that does not contain reflections.
If $\C^{n+1}\slash G$ is not weakly-exceptional, then there~is a
$\bar{G}$-invariant, irreducible, normal, Fano type projectively
normal subvariety $V\subset\P^n$ such that
$$
\mathrm{deg}\big(V\big)\leqslant {n\choose \mathrm{dim}\big(V\big)},%
$$
and for every $i\geqslant 1$ and for every $m\geqslant 0$, we have
$h^{i}(\mathcal{O}_{\P^n}(m)\otimes
\mathcal{I}_{V})=h^{i}(\mathcal{O}_{V}(m))=0$, and
\begin{equation}
\label{equation:Nadel-inequality}
h^{0}\Big(\mathcal{O}_{\P^n}\Big(\big(\mathrm{dim}(V)+1\big)\Big)\otimes \mathcal{I}_{V}\Big)\geqslant {n\choose \mathrm{dim}\big(V\big)+1},%
\end{equation}
where $\mathcal{I}_{V}$ is the ideal sheaf of the subvariety
$V\subset\P^n$. Let $\Pi$ be a general linear subspace in $\P^n$
of codimension $k\leqslant\mathrm{dim}(V)$. Put $X=V\cap\Pi$. Then
$h^{i}(\mathcal{O}_{\Pi}(m)\otimes \mathcal{I}_{X})=0$ for every
$i\geqslant 1$ and $m\geqslant k$, where $\mathcal{I}_{X}$ is the
ideal sheaf of the subvariety $X\subset\Pi$. Moreover, if $k=1$
and $\mathrm{dim}(V)\geqslant 2$, then $X$ is irreducible,
projectively normal and $h^{i}(\mathcal{O}_{X}(m))=0$ for every
$i\geqslant 1$ and $m\geqslant 1$.
\end{theorem}

\begin{corollary}
\label{corollary:Nadel}  Let $G$ be a finite subgroup in
$\GL_{n}(\C)$ that does not contain reflections. Then
the~singularity $\C^n/G$ is weakly-exceptional if for every
irreducible $\bar{G}$-invariant subvariety $V\subset\P^{n}$ there
exists no hypersurface in $\P^n$ of degree $\mathrm{dim}(V)+1$
that contains $V$.
\end{corollary}

Apart from their clear geometric nature,
Theorem~\ref{theorem:weakly-exceptional-quotient-criterion} and
Corollary~\ref{corollary:Nadel} are easy to apply in many cases
(but not always, since
Theorem~\ref{theorem:weakly-exceptional-quotient-criterion} is not
a criterion). In Section~\ref{section:series}, we will use them to
construct many explicit infinite series of Gorenstein weakly-exceptional
quotient singularities. In particular, we will use
Corollary~\ref{corollary:Nadel} to prove that infinitely many Gorenstein
weakly-exceptional quotient singularities exist in any dimension
(see Theorem~\ref{theorem:infinite} and
Corollary~\ref{corollary:infinitely-many}) and to prove

\begin{theorem}
\label{theorem:Heisenberg} Let $p\ge 3$ be a prime number,
and let $G$ be a subgroup in $\SL_p(\C)$ that is
isomorphic to the Heisenberg group of order $p^3$. Then
$\C^{p}\slash G$ is weakly-exceptional.
\end{theorem}

We prove Theorem~\ref{theorem:Heisenberg} (after stating it a bit
more explicitely) in Section~\ref{section:series} (see
Theorem~\ref{theorem:Geisenberg}).
Theorem~\ref{theorem:Heisenberg} gives many examples of weakly
exceptional singularities in dimension $p$ corresponding to the
groups containing the group $G$ (cf.
Theorem~\ref{theorem:primitive-5}). Another reason to study weak
exceptionality of the quotient by the Heisenberg group is that the
order of this group is relatively small compared to that of the
groups considered in Theorem~\ref{theorem:infinite}

\smallskip
We have already seen that being weakly-exceptional is not so easy to
check for a higher dimensional quotient singularity. This is not
surprising: the life is not easy in higher dimensions.
Surprisingly, the life is still easy in dimension five as it
follows from

\begin{theorem}[{cf. Theorem~\ref{theorem:Vanya-Kostya-SL-4-weakly-exceptional}}]
\label{theorem:main-5}  Let $G$ be a finite subgroup in
$\GL_{5}(\C)$ that does not contain reflections. Then
the~singularity $\C^5/G$ is weakly-exceptional if and only if
the~group $G$ is transitive and does  not have semi-invariants of
degree at most $4$.
\end{theorem}

However, in dimension six, nature takes its revenge.
In this case we can only prove a sufficient condition.

\begin{theorem}
\label{theorem:main-6}  Let $G$ be a finite subgroup in
$\GL_{6}(\C)$ that does not contain reflections. Then
the~singularity $\C^6/G$ is weakly-exceptional if the~following
five conditions are satisfied:
\begin{enumerate}
\item the~group $G$ is transitive,%
\item the~group $G$ does not have semi-invariants of degree at most $5$,%
\item there is no irreducible $\bar{G}$-invariant smooth rational
cubic scroll\footnote{This is just $\mathbb{P}^{1}\times\mathbb{P}^{2}$
embedded by Segre.} in~$\P^5$,%
\item there is no irreducible $\bar{G}$-invariant
complete intersection of two quadric hypersurfaces in~$\mathbb{P}^{5}$,%
\item there is no irreducible $\bar{G}$-invariant, normal,
projectively normal, non-degenerate, Fano type threefold
$X\subset\mathbb{P}^{5}$ with at most rational singularities of
degree $6$ and sectional genus~$3$ such that
$h^{0}(\mathcal{O}_{\mathbb{P}^{5}}(2)\otimes\mathcal{I}_{X})=0$
and
$h^{0}(\mathcal{O}_{\mathbb{P}^{5}}(3)\otimes\mathcal{I}_{X})=4$
(cf. Remark~\ref{remark:Bordiga-threefold} below).
\end{enumerate}
\end{theorem}

The main purpose of this paper is to prove
Theorems~\ref{theorem:weakly-exceptional-quotient-criterion},
\ref{theorem:main-5} and \ref{theorem:main-6}, which is done in
Sections~\ref{section:criterion}, \ref{section:5-dim} and
\ref{section:6-dim}, respectively. The proofs of
Theorems~\ref{theorem:main-5} and~\ref{theorem:main-6} are based
on Theorem~\ref{theorem:weakly-exceptional-quotient-criterion} and
appear to be naturally linked to many classical geometric
constructions, which are very interesting on their own. For
example, in the proof of Theorem~\ref{theorem:main-5} we naturally
come across various surfaces of small degree in~$\P^4$, in
particular singular Bordiga surfaces.

It should be pointed out that Theorem~\ref{theorem:main-5} is
indeed a criterion, while
Theorems~\ref{theorem:weakly-exceptional-quotient-criterion}
and~\ref{theorem:main-6} provide us only sufficient conditions to
be weakly-exceptional. One can easily see that
Theorem~\ref{theorem:weakly-exceptional-quotient-criterion} is
quite far from being a criterion by comparing it with
Theorems~\ref{theorem:weakly-exceptional-dim-3} and
\ref{theorem:Vanya-Kostya-SL-4-weakly-exceptional}. On the other
hand, the first four conditions in Theorem~\ref{theorem:main-6}
are also necessary conditions for the weak-exceptionality (cf.
\cite[Lemma~3.21]{ChSh09}). Unfortunately, we do not know whether
the fifth  condition in Theorem~\ref{theorem:main-6} is really
necessary or not.

\begin{remark}
\label{remark:Bordiga-threefold} Let $X$ be a projectively normal
non-degenerate threefold in $\mathbb{P}^{5}$ of degree $6$ and
sectional genus $3$ that has at most Kawamata log terminal
singularities. If $X$ is smooth, then it follows from
\cite[Proposition~1.7]{Okonek82} that $X$ is a so-called smooth
Bordiga scroll, which is a projectivization of a two-dimensional
stable vector bundle $\mathcal{E}$ on $\P^2$ such that
$\mathrm{c}_{1}(\mathcal{E})=0$ and
$\mathrm{c}_{2}(\mathcal{E})=0$. Smooth Bordiga scrolls have been
studied in \cite{Okonek82}, \cite{Ottaviani92} and
\cite{MezzettiPortelli99}. One can show that smooth Bordiga
scrolls are weak Fano threefolds, i.e. their anticanonical
divisors are big and nef. Note that smooth Bordiga scrolls are
missing in the classifications obtained
in~\cite{JahnkePeternellRadloff2005}
and~\cite{JahnkePeternellRadloff2011}. Smooth hyperplane sections
of smooth Bordiga scrolls are known as smooth Bordiga surfaces
(see \cite{Bordiga1887}, \cite[Chapter~XIV]{Room1938}), which can
be obtained by blowing up $\P^2$ at $10$ sufficiently general
points. It follows from \cite{Rathmann89} that some smooth Bordiga
surfaces in $\mathbb{P}^{4}$ are set-theoretic intersections of a
cubic and a quartic hypersurfaces (see also
\cite[Proposition~19]{HulekOkonekVandeVen86}). If one can show
that $X$ is also a set-theoretic intersections of a cubic and a
quartic hypersurfaces in $\P^5$ (without imposing any additional
assumption on $X$ except may be those that are used in the fifth
condition in Theorem~\ref{theorem:main-6}), then the fifth
condition in Theorem~\ref{theorem:main-6} is also a necessary
condition for the singularity $\C^6/G$ in
Theorem~\ref{theorem:main-6} to be weakly-exceptional.
Unfortunately, we do not even know whether there exists an example
of a smooth three-dimensional Bordiga scroll that is a
set-theoretic intersections of a cubic and a quartic
hypersurfaces. On the other hand, there is a small chance that the
fifth condition in Theorem~\ref{theorem:main-6} follows from the
first four, which would imply that it can be dropped from
Theorem~\ref{theorem:main-6}.
\end{remark}

By Theorem~\ref{theorem:weakly-exceptional-quotient}, to apply
Theorems~\ref{theorem:main-5} and \ref{theorem:main-6}, we may
assume that $G\subset\SL_{n+1}(\mathbb{C})$, since there exists
a~finite subgroup $G^{\prime}\subset\SL_{n+1}(\mathbb{C})$
such~that $\phi(G^{\prime})=\bar{G}$. Every transitive finite
subgroup of $\SL_2(\mathbb{C})$ gives rise to a weakly-exceptional
singularity by Theorem~\ref{theorem:weakly-exceptional-quotient}.
In~\cite{Sak10}, Sakovich used
Theorems~\ref{theorem:weakly-exceptional-dim-3} and
\ref{theorem:Vanya-Kostya-SL-4-weakly-exceptional} to obtain
an~explicit classification of the transitive finite subgroups in
$\SL_3(\mathbb{C})$ and $\SL_4(\mathbb{C})$ corresponding to
three-dimensional and four-dimensional weakly-exceptional quotient
singularities, respectively. Probably, a similar classification is
possible in dimensions~$5$ and~$6$ using
Theorems~\ref{theorem:main-5} and \ref{theorem:main-6},
respectively. But this task requires huge amount of computations
that goes beyond the scope and the purpose of this paper. So
instead, let us apply Theorems~\ref{theorem:main-5} and
\ref{theorem:main-6} to classify all primitive (see
\cite[Definition~1.21]{ChSh09}) finite subgroups in
$\SL_5(\mathbb{C})$ and $\SL_6(\mathbb{C})$ corresponding to
five-dimensional and six-dimensional weakly-exceptional quotient
singularities, respectively.

\begin{theorem}[{cf.~\cite[Theorems~1.22 and~5.6]{ChSh09}}]
\label{theorem:primitive-5} Let $G\subset\SL_5(\C)$ be a finite
primitive group. Then the singularity $\C^5\slash G$ is
weakly-exceptional if and only if $G$ contains the Heisenberg
group of order $125$.
\end{theorem}

\begin{proof}
The ``if'' part follows immediately by
Theorem~\ref{theorem:Heisenberg}. Let us prove the ``only if''.
Suppose that $G$ does not contain the Heisenberg group of order
$125$. By Theorem~\ref{theorem:weakly-exceptional-quotient}, the
weak-exceptionality of the singularity $\C^5\slash G$ depends only
on the image of the group $G$ in $\PSL_5(\C)$, so that we may
assume that $Z(G)\subset [G,G]$ (see \cite{Fe71}). Then the group
$G$ is one of the following groups: $\A_5$, $\A_6$, $\SS_5$,
$\SS_6$, $\PSL_2(\mathbb{F}_{11})$ or $\PSp_4(\F_3)$ (see \cite[\S
8.5]{Fe71}). In all of these cases there exists a semi-invariant
of $G$ of degree at most~$4$ by~\cite[Lemma~5.3]{ChSh09}, so that
the singularity $\C^5\slash G$ is not weakly-exceptional by
Theorem~\ref{theorem:weakly-exceptional-quotient}.
\end{proof}

\begin{theorem}[{cf.~\cite[Theorem~1.14]{ChSh10}}]
\label{theorem:primitive-6} Let $G$ be a finite primitive subgroup
in $\SL_6(\C)$. Then the singularity $\C^6\slash G$ is
weakly-exceptional if and only if there exists a~lift of
the~sub\-group~$\bar{G}\subset\PGL_6(\C)$ to $\SL_6(\C)$ that is
contained in the~following list\footnote{We label the cases
according to the notation of~\cite{Fe71} and \cite{Li71}.}:
\begin{itemize}
\item[(V)] $6.\A_6$,%
\item[(VIII)] $6.\A_7$,%
\item[(XIV)]
\begin{itemize}
\item[(i)] $\SU_3(\F_3)$, %
\item[(ii)] an~extension of the~subgroup described in XIV(i) by an automorphism of order $2$,%
\end{itemize}
\item[(XV)]
\begin{itemize}
\item[(i)] $6.\PSU_4(\F_3)$, %
\item[(ii)] an~extension of the~subgroup described in XV(i) by an automorphism of order $2$,%
\end{itemize}
\item[(XVI)] $2.\HaJ$, where $\HaJ$ is the~Hall--Janko sporadic simple group,%
\item[(XVII)]
\begin{itemize}
\item[(i)] $6.\PSL_3(\F_4)$, %
\item[(ii)] an~extension of the~subgroup described
in XVII(i) by an automorphism of order $2$.%
\end{itemize}
\end{itemize}
\end{theorem}

\begin{proof}
The classification of the primitive subgroups of $\SL_6(\C)$ is
given in~\cite[{\S 3}]{Li71}. Browsing through it, we find that
primitive subgroups of $\SL_6(\C)$ that do not have
semi-invariants of degree at most $5$ are exactly those listed in
the assertion of the theorem (see the proof
of~\cite[Theorem~3.3]{ChSh10}) and those that satisfy the
hypotheses of \cite[Lemma~3.24]{ChSh09}. If $G$ satisfies the
hypotheses of \cite[Lemma~3.24]{ChSh09}, then the singularity
$\C^6\slash G$ is not weakly-exceptional by
Theorem~\ref{theorem:weakly-exceptional-quotient} and
\cite[Lemma~3.24]{ChSh09}. If $G$ is of type XIV(i), XIV(ii),
XV(i), XV(ii), XVI, XVII(i) and~XVII(ii), then $G$ does not have
an irreducible representation $W$ such that \mbox{$2\le\dim(W)\le
4$} (see~\cite{Atlas}), which easily implies that these cases give
rise to weakly-exceptional singularities by
Theorem~\ref{theorem:main-6}. If $G$ is of type~V, then $G\cong
6.\A_6$ has no irreducible two-dimensional representations
(see~\cite{Atlas}), and no irreducible representation of the group
$G$ of dimension $3$ or $4$ is contained in $\Sym^3(V^{\vee})$,
where $V\cong\C^6$ is the $G$-representation in question. The
latter can be checked by a direct computation (we used the Magma
software~\cite{Magma} to carry it out). Therefore, the singularity
$\C^6\slash G$ is again weakly-exceptional by
Theorem~\ref{theorem:main-6}.
\end{proof}

It seems possible to apply Theorems~\ref{theorem:Heisenberg},
\ref{theorem:main-5}, \ref{theorem:main-6},
\ref{theorem:primitive-5}, and \ref{theorem:primitive-6} to
construct non-conjugate isomorphic finite subgroups in Cremona
groups of high ranks (cf. \cite[Example~6.5]{Ch07b}). For example,
if $\mathrm{lct}(\P^n, \bar{G})\geqslant 1$ and $\P^n$ is
$\bar{G}$-birationally super-rigid (see
\cite[Definition~6.1]{Ch07b}), then it follows from
\cite[Theorem~6.4]{Ch07b} that there exists no
$\bar{G}\times\bar{G}$-equivariant birational map
$\P^n\times\P^n\dasharrow\mathbb{P}^{2n}$  with respect to the
product action of the group  $\bar{G}\times\bar{G}$ on
$\P^n\times\P^n$. By
Theorem~\ref{theorem:weakly-exceptional-criterion}, we can use
Theorems~\ref{theorem:Heisenberg}, \ref{theorem:main-5},
\ref{theorem:main-6}, \ref{theorem:primitive-5}, and
\ref{theorem:primitive-6} to obtain many finite subgroups
$\bar{G}\subset\mathrm{Aut}(\P^n)$ with $\mathrm{lct}(\P^n,
\bar{G})\geqslant 1$. However, if $n\geqslant 3$, then it is
usually very hard to prove that $\P^n$ is $\bar{G}$-birationally
super-rigid (cf. \cite{ChSh09A6}). In fact, we do not know any
such example if $n\geqslant 4$.

\smallskip

Let us describe the structure of this paper. In
Section~\ref{section:criterion}, we prove
Theorem~\ref{theorem:weakly-exceptional-quotient-criterion}.  In
Section~\ref{section:series}, we construct (see
Theorem~\ref{theorem:infinite}) infinite series of Gorenstein
weakly-exceptional quotient singularities in any dimension (in
particular, proving that weakly-exceptional quotient singularities
exist in all dimensions), and prove
Theorem~\ref{theorem:Heisenberg}. In Section~\ref{section:5-dim},
we prove Theorem~\ref{theorem:main-5}. In
Section~\ref{section:6-dim}, we prove
Theorem~\ref{theorem:main-6}.  In
Appendix~\ref{section:septic-curves}, we prove an auxiliary
statement concerning smooth irreducible curves in~$\P^3$ of
genus~$5$ and degree~$7$: we prove that any such curve is a
set-theoretic intersection of cubics (this result might be known
to experts, but we did not manage to find a reference in the
literature).

\smallskip

Throughout the proofs of Theorems~\ref{theorem:main-5}
and~\ref{theorem:main-6} we often work with singular del Pezzo
surfaces. A good preliminary reading for the corresponding
techniques may be~\cite{Demazure}.
On the other hand, the facts from the
theory of representations of finite groups we will need
(say, in the proof of Theorem~\ref{theorem:Geisenberg})
do not go beyond the elementary material that can
be found in any textbook on the subject
(for example,~\cite{JPSerre}).

\smallskip

The problem of finding a nice geometric criterion for a
five-dimensional quotient singularity to be weakly-exceptional
originated during the first author participation in the 18th
G\"okova Conference in Turkey. The first author would like to
thank Selman Akbulut for inviting him to this beautiful place.
The~authors would like to thank Marco Andreatta, Eduardo Ballico,
Pietro De Poi, Igor Dolgachev, Stephane Lamy, Jihun Park, Emilia
Mezzetti, Yuri Prokhorov and Franchesco Russo for many fruitful
discussions.

We proved both Theorems~\ref{theorem:main-5} and
\ref{theorem:main-6} while participating in the \emph{Research in
Groups} program in the Center of International Research in
Mathematics (Trento, Italy). We finished this paper at the
Institute for the Physics and Mathematics of the Universe (Tokyo,
Japan). We are really grateful to CIRM and IPMU for the beautiful
working conditions. Special thanks goes to Sergey Galkin for his
warm and encouraging support during our stay at IPMU. The work was
also supported by the grants NSF DMS-1001427, N.Sh.-4713.2010.1,
RFFI 11-01-00336-a, RFFI 11-01-92613-KO-a, RFFI 08-01-00395-a,
RFFI 11-01-00185-a, and by AG Laboratory GU-HSE, RF government
grant 11~11.G34.31.0023.

\section{Weak-exceptionality criterion}
\label{section:criterion}

The purpose of this section is to prove
Theorem~\ref{theorem:weakly-exceptional-quotient-criterion}. Let
$G$ be a~finite subgroup in $\GL_{n+1}(\mathbb{C})$, and let
$\phi\colon\GL_{n+1}(\mathbb{C})\to\PGL_{n+1}(\mathbb{C})$ be
the~natural projection. Put $\bar{G}=\phi(G)$. Let us identify
$\PGL_{n+1}(\mathbb{C})$ with $\mathrm{Aut}(\mathbb{P}^{n})$. Let
us denote by $H$ a general hyperplane in $\P^n$. Suppose that $G$
contains no reflections and $\C^{n+1}\slash G$ is not
weakly-exceptional. Then $\mathrm{lct}(\P^n,\bar{G})<1$ by
Theorem~\ref{theorem:weakly-exceptional-quotient}. Thus, there is
an~effective $\bar{G}$-invariant $\mathbb{Q}$-divisor $D$ on
$\mathbb{P}^{n}$ such that $D\sim_{\mathbb{Q}}(n+1)H$ and
a~positive rational number $\lambda<1$ such~that
$(\mathbb{P}^{n},\lambda D)$ is strictly log canonical.

Let $V$ be a~minimal center of log canonical singularities of the
log pair $(\mathbb{P}^{n},\lambda D)$ (see
\cite[Definition~1.3]{Kaw97}, \cite{Kaw98}). Then it follows from
the Kawamata subadjunction theorem (see \cite[Theorem~1]{Kaw98})
that $V$ is normal and has at most rational singularities, and for
any rational number $\epsilon<(1-\lambda)(n+1)$ there is
an~effective
$\mathbb{Q}$-divisor~$\Delta_{\epsilon}$~on~$V$~such~that
$-(K_{V}+\Delta_{\epsilon})\sim_{\mathbb{Q}} \epsilon H\vert_{V}$,
and the log pair $(V,\Delta_{\epsilon})$ has Kawamata log terminal
singularities. We see that $V$ is a Fano type subvariety. In
particular, we see that $h^{i}(\mathcal{O}_{V}(mH\vert_V))=0$ for
every $i\geqslant 1$ and $m\geqslant 0$ by the Nadel--Shokurov
vanishing theorem (see~\cite[Theorem~2.16]{Ko97}
or~\cite[Appendix]{Ambro-ladders}).

Let $Z$ be the $\bar{G}$-orbit of the subvariety $V$, and let
$\mu$ be a rational number such that $\mu>1$ and $\mu\lambda<1$.
Then it follows from \cite[Lemma~2.8]{ChSh09} that there exists
an~effective~$\bar{G}$-in\-va\-riant $\mathbb{Q}$-divisor
$D^{\prime}$~on~$\P^n$~such~that $D^{\prime}\sim_{\mathbb{Q}}
\mu\lambda D$, the~log pair $(\P^n, D^{\prime})$ is strictly log
canonical, and every log canonical center of the log pair $(\P^n,
D^{\prime})$ is a component of $Z$. Since $\mu\lambda<1$, it
follows from the Nadel--Shokurov vanishing theorem that
$h^{i}(\mathcal{I}_{Z}\otimes\mathcal{O}_{\P^n}(m))=0$ for every
$i\geqslant 1$ and $m\geqslant 0$, where $\mathcal{I}_{Z}$ is an
ideal sheaf of $Z$.  Thus, the sequence of cohomology groups
$$
0\longrightarrow H^{0}\Big(\mathcal{O}_{\P^n}\big(m\big)\otimes\mathcal{I}_{Z}\Big)\longrightarrow H^{0}\Big(\mathcal{O}_{\P^n}\big(m\big)\Big)\longrightarrow H^{0}\Big(\mathcal{O}_{V}\big(mH\vert_Z\big)\Big)\longrightarrow 0%
$$
is exact for every $m\geqslant 0$, which implies $Z$ is connected,
and $V$ is projectively normal if $V=Z$. Since $Z$ is, one has
$V=Z$ by \cite[Proposition~1.5]{Kaw97}.

Put $d=\dim(V)$. Let $\Pi$ be a~general linear subspace in
the projective space $\mathbb{P}^{n}$ of codimension $k\leqslant d$. Put
$D_{\Pi}^{\prime}=D^{\prime}\vert_{\Pi}$. Then every center of log
canonical singularities ofthe log pair $(\Pi,D_{\Pi}^{\prime})$ is
a components of the intersection $V\cap\Pi$, since $\Pi$ is
sufficiently general. Put $H_{\Pi}=H\vert_{\Pi}$ and $X=\Pi\cap
V$. Then
$$K_{\Pi}+D_{\Pi}^{\prime}\sim_{\mathbb{Q}}
(\mu\lambda\big(n+1\big)-n+k-1)H_{\Pi},$$
where $\mu\lambda<1$. Thus, it follows from the Nadel--Shokurov
vanishing theorem  that
\begin{equation}
\label{equation:vanishin-cut}
h^{i}\Big(\mathcal{O}_{\Pi}\big(mH_{\Pi}\big)\otimes \mathcal{I}_{X}\Big)=0%
\end{equation}
for every $i\geqslant 1$ and $m\geqslant k$, where
$\mathcal{I}_{X}$ is the ideal sheaf of the subvariety
$X\subset\Pi$. If we put $k=d$, then it follows from
$(\ref{equation:vanishin-cut})$ that
$$
\mathrm{deg}\big(V\big)=\big|V\cap \Pi\big|=h^{0}\big(\mathcal{O}_{X}\big)=h^{0}\Big(\mathcal{O}_{\Pi}\big(dH_{\Pi}\big)\Big)-h^{0}\Big(\mathcal{O}_{\Pi}\big(H_{\Pi}\big)\otimes\mathcal{I}_{X}\Big)\leqslant{n\choose d}.%
$$

Suppose that $k=1$ and $d\geqslant 2$. Then $X$ is irreducible and
the sequence
$$
0\longrightarrow H^{0}\Big(\mathcal{O}_{\Pi}\big(mH_{\Pi}\big)\otimes\mathcal{I}_{X}\Big)\longrightarrow H^{0}\Big(\mathcal{O}_{\Pi}\big(mH_{\Pi}\big)\Big)\longrightarrow H^{0}\Big(\mathcal{O}_{X}\big(mH\vert_X\big)\Big)\longrightarrow 0%
$$
is exact for every $m\geqslant k=1$ by
$(\ref{equation:vanishin-cut})$, which implies that $X\subset\Pi$
is projectively normal.

It follows from $(\ref{equation:vanishin-cut})$ that
$h^{i}(\mathcal{O}_{X}(mH\vert_X))=h^{i}(\mathcal{O}_{\Pi}(mH_{\Pi}))=0$
for every $i\geqslant 1$ and $m\geqslant 1$, since
$h^{i}(\mathcal{O}_{\Pi}(mH_{\Pi}))=0$ for every $i\geqslant 1$
and $m\geqslant 0$.

Let $\Lambda$ be a general linear subspace in $\P^n$ of
codimension $d+1$, let $N$ be a very big~integer,
and let $H_{1},H_{2},\ldots,H_{N}$ be sufficiently general
hyperplanes in $\P^n$ that contain $\Lambda$. Then
\begin{equation}
\label{equation:log-pair-nadel}
\Big(\P^n, D^{\prime}+\frac{d+1}{N}\sum_{i=1}^{N}H_{i}\Big)%
\end{equation}
is strictly log canonical. Moreover, it follows from the
construction of the log pair $(\ref{equation:log-pair-nadel})$
that the only centers of log canonical singularities of the log
pair $(\ref{equation:log-pair-nadel})$ are $V$ and $\Lambda$. Note
that $V\cap\Lambda=\varnothing$ by construction. Put $m=d+1$. Then
it follows from the Nadel--Shokurov vanishing theorem that
$$
h^{0}\Big(\mathcal{O}_{\P^n}\big(m\big)\otimes \mathcal{I}_{V\cup\Lambda}\Big)={n+m\choose m}-h^{0}\Big(\mathcal{O}_{V}\big(mH\vert_V\big)\Big)-{n\choose m},%
$$
since $V\cap\Lambda=\varnothing$. Thus, it follows from the
projective normality of the subvariety $V\subset\P^n$ that
$$
h^{0}\Big(\mathcal{O}_{\P^n}\big(m\big)\otimes \mathcal{I}_{V}\Big)={n+m\choose m}-h^{0}\Big(\mathcal{O}_{V}\big(mH\vert_V\big)\Big)={n\choose m}+h^{0}\Big(\mathcal{O}_{\P^n}\big(m\big)\otimes\mathcal{I}_{V\cup\Lambda}\Big),%
$$
which implies the inequality $(\ref{equation:Nadel-inequality})$
(cf. the proof of~\cite[Theorem~6.1]{Na90}) and completes the
proof of
Theorem~\ref{theorem:weakly-exceptional-quotient-criterion}.

\section{Infinite series}
\label{section:series}

Let $G$ be a~finite subgroup in $\GL_{n+1}(\mathbb{C})$, and let
$\phi\colon\GL_{n+1}(\mathbb{C})\to\PGL_{n+1}(\mathbb{C})$ be
the~natural projection. Put $\bar{G}=\phi(G)$ and
$V=\mathbb{C}^{n+1}$, and let us identify $\PGL_{n+1}(\mathbb{C})$
with $\Aut(\P^n)$. Let us fix
coordinates $(x_{0},\ldots,x_{n})$ on $V$, and let us use them
also as projective coordinates on $\P^n$.

\begin{theorem}\label{theorem:infinite}
Suppose that the group $G$ does not contain reflections, the
representation $V$ of the group $G$ is irreducible, and $G$
contains a subgroup $\Gamma\cong\Z_k^n$ for some integer
$k\geqslant n+1$ such that $\Gamma$ is generated by the
transformations $\gamma_{1},\gamma_{2},\ldots,\gamma_{n}$ defined
by
$$
\gamma_i\colon(x_0, x_1, \ldots, x_{i-1}, x_i, x_{i+1}, \ldots, x_n)\mapsto (\zeta^{-1}x_0, x_1,  \ldots, x_{i-1}, \zeta x_i, x_{i+1}, \ldots, x_n),%
$$
where $\zeta$ is a primitive root of unity of degree $k$. Then
the singularity $\C^{n+1}\slash G$ is weakly exceptional.
\end{theorem}

\begin{proof}
Suppose that the singularity $\C^{n+1}\slash G$ is not weakly
exceptional. Then it follows from Corollary~\ref{corollary:Nadel}
that there exists a $\bar{G}$-invariant irreducible subvariety
$X\subset\P^n$ such that there exists a hypersurface in $\P^n$ of
degree $\mathrm{dim}(X)+1$ that contains $X$. Let us derive a
contradiction. To do this, we may swap $G$ with any other group
$G^{\prime}\subset\GL_{n+1}(\C)$ such that
$\phi(G^{\prime})=\bar{G}$. Thus, adding a scalar matrix
$\mathrm{diag}(\zeta, \ldots, \zeta)$ to the group $G$, we may
assume that $G$ contains a subgroup
$\Gamma^{\prime}\cong\Z_k^{n+1}$ that is generated by the
transformations
$\gamma_{0}^{\prime},\gamma_{1}^{\prime},\gamma_{2}^{\prime},\ldots,\gamma_{n}^{\prime}$
such that
$$
\gamma^{\prime}_i\colon(x_0, \ldots, x_{i-1}, x_i, x_{i+1}, \ldots, x_n)\mapsto (x_0,  \ldots, x_{i-1}, \zeta x_i, x_{i+1}, \ldots, x_n),%
$$
where as before $\zeta$ is a primitive root of unity of degree
$k$. This assumption is not crucial for the proof, but it makes
some steps clearer.

Put $d=\dim(X)+1$. Then $d\leqslant n$. Every
$\Gamma^{\prime}$-invariant vector subspace in
$\C[x_{0},\ldots,x_n]$ consisting of forms of degree $d$ splits
into a sum of pairwise non-isomorphic one-dimensional
representations of the group $\Gamma^{\prime}$ that are generated
by monomials. Thus, any non-trivial $\Gamma^{\prime}$-invariant
vector subspace in $\C[x_{0},\ldots,x_n]$ consisting of forms of
degree $d$ must contain a monomial. In particular, every
semi-invariant form in $\C[x_{0},\ldots,x_n]$ of degree $d$ with
respect to $\Gamma^{\prime}$ must be a monomial. Therefore,  the
group $G$ does not have semi-invariant form in
$\C[x_{0},\ldots,x_n]$ of degree~$d$, because~$G$ is transitive.
Indeed, if there exists a semi-invariant form in
$\C[x_{0},\ldots,x_n]$ of degree~$d$ with respect to~$G$, then
this form must be a monomial $x_{i_1}x_{i_2}\ldots x_{i_d}$ for
some (not necessarily distinct) $i_{1},i_{2},\ldots,i_{d}$ in
$\{0,\ldots,n\}$, which implies that the vector subspace in $V$
given by $x_{i_1}=x_{i_2}=\ldots=x_{i_d}=0$ must be $G$-invariant
and proper, since $d\leqslant n$.

Let $W(X)$ be the vector subspace in $\C[x_{0},\ldots,x_n]$
consisting of forms of degree $d$ that vanish on $X$. Then $W(X)$
is non-zero and $G$-invariant, because $X$ is $\bar{G}$-invariant.
Thus, the vector subspace $W(X)$ must contain a monomial
$x_{j_1}x_{j_2}\ldots x_{j_d}$ for some $j_{1},j_{2},\ldots,j_{d}$
in $\{0,\ldots,n\}$. Therefore, the subvariety $X$ is contained in
the union of hyperplanes in $\P^n$ that are given by $x_{j_1}=0,
x_{j_2}=0, \ldots, x_{j_d}=0$. Since $X$ is irreducible, the
subvariety $X$ is contained in one of these hyperplanes. Without
loss of generality, we may assume that $X$ is contained in the
hyperplane that is given by $x_{j_1}=0$. Therefore, the linear
span of the subvariety $X$ is also contained in the hyperplane
that is given by $x_{j_1}=0$, which is impossible, since $V$ is an
irreducible representation of the group $G$.
\end{proof}

\begin{corollary}
\label{corollary:infinitely-many} For any  $N>1$ there are
infinitely many Gorenstein weakly-exceptional quotient singularities of
dimension $N$.
\end{corollary}

In the remaining part of this section we will give a proof of
Theorem~\ref{theorem:Heisenberg}. In order to do this, we need the
following

\begin{lemma}\label{lemma:p-adic}
Let $p\ge 2$ be a prime number, let $P(x)$ be a
polynomial in $\Q[x]$ of degree $d$. Put
$$
P(x)=\sum_{i=0}^d \frac{b_i}{c_{i}} x^i,
$$
where $b_i$ and $c_i$ integers such that $c_{i}\ne 0$ and
$\mathrm{gcd}(b_i,c_i)=1$. Take $\gamma\in\Z$.
For~$i\in\{0,\ldots,d\}$, let $P(\gamma+i)=\frac{r_i}{q_i}$, where
$r_i$ and $q_i$ are integers such that $q_{i}\ne 0$ and
$\mathrm{gcd}(r_i,q_i)=1$. Suppose that $r_0\equiv
r_1\equiv\ldots\equiv r_{d}\equiv 0\ \mathrm{mod}\ p$ and $d<p$.
Then $b_0\equiv b_1\equiv \ldots\equiv b_d\equiv 0\ \mathrm{mod}\
p$.
\end{lemma}

\begin{proof}
For~every $i$, put $c_{i}=p^{t_i}m_{i}$, where $t_{i}$ is a
non-negative integer, and $m_{i}$ is an integer that is not
divisible by $p$. Put $t=\mathrm{max}(t_{0},\ldots,t_{d})$, and
put $N=\mathrm{lcm}(c_0,\ldots,c_d)$. Then one has
$N=p^t\mathrm{lcm}(m_0,\ldots,m_d)$ and $Nf(x)\in\mathbb{Z}[x]$.
Take any $k\in\{0,\ldots,d\}$ such that $t=t_{k}$. Then
$$
Nf\big(\gamma\big)\equiv \frac{Nr_0}{q_0}\equiv  Nf\big(\gamma+1\big)\equiv \frac{Nr_1}{q_1}\equiv \ldots\equiv Nf\big(\gamma+d\big)\equiv\frac{Nr_d}{q_d}\equiv 0\ \mathrm{mod}\ p,%
$$
which immediately implies that all integers
$Nb_{0}/c_0,Nb_{1}/c_1,\ldots,Nb_{d}/c_d$ must be divisible by
$p$, because only the zero polynomial in $\mathbb{F}_{p}[x]$ of
degree $d<p$ has $d+1$ different roots. Then
$$
\frac{Nb_{k}}{c_{k}}=\frac{p^{t_{k}}\mathrm{lcm}\big(m_0,\ldots,m_d\big)b_{k}}{p^{t^{k}}m_{k}}=\frac{\mathrm{lcm}\big(m_0,\ldots,m_d\big)b_{k}}{m_{k}}
$$
is divisible by $p$, which implies that $b_{k}$ is divisible by
$p$. Since
$\mathrm{gcd}(b_{k},p^{t_k}m_k)=\mathrm{gcd}(b_k,c_k)=1$, we see
that $t=t_{k}=0$. Hence, we have $t_{0}=\ldots=t_{d}=0$. Thus, we
see that $N$ is not divisible by $p$, which implies that
$b_0\equiv b_1\equiv \ldots\equiv b_d\equiv 0\ \mathrm{mod}\ p$,
since we know that all integers
$Nb_{0}/c_0,Nb_{1}/c_1,\ldots,Nb_{d}/c_d$ are divisible by $p$.
\end{proof}

Now we are ready to prove

\begin{theorem}
\label{theorem:Geisenberg}
Let $p\ge 3$ be a prime number.
Suppose that $n+1=p$ and $G$ is generated by
the elements $(x_0:x_1:\ldots:x_{p-1})\to
(x_1:\ldots:x_{p-1}:x_0)$ and
$$
(x_0:x_1:\ldots:x_{p-1})\mapsto (x_0:\zeta x_1:\ldots:\zeta^{p-1}x_{p-1}),
$$
where $\zeta$ is a primitive root of unity of degree $p$. Then
the singularity $\C^{n+1}\slash G$ is weakly exceptional.
\end{theorem}

\begin{proof}
Suppose that the singularity $\C^{n+1}\slash G$ is not weakly
exceptional. Then it follows from
Theorem~\ref{theorem:weakly-exceptional-quotient-criterion}  that
there exists a $\bar{G}$-invariant irreducible normal Fano type
subvariety $V\subset\P^n$ such that
$$
\mathrm{deg}\big(V\big)\leqslant {n\choose \mathrm{dim}\big(V\big)},%
$$
and $h^i(V,\O_V(m))=0$ for any $i\ge 1$ and any $m\ge 0$, where
$\O_V(m)=\O_{V}\otimes\O_{\P^n}(m)$. Then
\begin{equation}\label{eq:vanishing-for-Heisenberg}
h^0(V,\O_V(m))=h^0(\P^{n}, \O_{\P^{n}}(m))-h^0(\P^{n}, \mathcal{I}_V(m)),%
\end{equation}
where $\mathcal{I}_{V}$ is the ideal sheaf of $V$.

Let $Z(G)$ be the center of the group $G$. Then $Z(G)\cong\Z_p$.
It is well-known that any irreducible representation of $G$ with a
non-trivial action of the center $Z(G)$ is $p$-dimensional. In
particular, the group $G$ has no semi-invariants of degree less
than $p$, which implies that $\dim(V)\le p-2$.

Put $W_m=H^0(\P^{n}, \mathcal{I}_V(m))$. Then $W_{m}$ is a linear
representation of the group $G$, and $Z(G)$ acts non-trivially on
$W_m$ for every $m\in\{1,\ldots, p-1\}$. Therefore $\dim(W_m)$ is
divisible by $p$ for every $m\in\{1,\ldots,p-1\}$.
Applying~$(\ref{eq:vanishing-for-Heisenberg})$ for every
$m\in\{1,\ldots,p-1\}$ and keeping in mind that $h^0(\P^{n},
\O_{\P^{n}}(m))$ is divisible by $p$ for every
$m\in\{1,\ldots,p-1\}$, we see that $h^0(V,\O_V(m))$ is divisible
by $p$ for every $m\in\{1,\ldots,p-1\}$. Put $d=\dim(V)$. Since
$h^0(V,\O_V(m))=\chi(V,\O_V(m))$ for every $m\geqslant 0$, there
are integers $a_{0},a_{1},\ldots,a_{d}$ such that
$$
h^0(V,\O_V(m))=a_dm^d+a_{d-1}m^{d-1}+\ldots+a_0
$$
for every $m\geqslant 0$. Applying Lemma~\ref{lemma:p-adic} to the
polynomial $P(m)=a_dm^d+a_{d-1}m^{d-1}+\ldots+a_0$, we see that
$a_0=b_0/c_0$, where $\mathrm{gcd}(b_0,c_0)=1$ and $b_0$ is
divisible by $p$. On the other hand,
applying~$(\ref{eq:vanishing-for-Heisenberg})$ for $m=0$, we
obtain that $a_0=h^0(V,\O_V)=1$, which is a contradiction.
\end{proof}

\section{Five-dimensional singularities}
\label{section:5-dim}

The main purpose of this section is to prove
Theorem~\ref{theorem:main-5}. We start with three easy
observations.

\begin{lemma}
\label{lemma:small-orbit-on-P1} Let $\bar{G}$ be a finite subgroup
in $\Aut(\P^1)$, and let $\Omega$ be a $\bar{G}$-orbit in $\P^1$.
If $|\Omega|\in\{1,2,5,7,8,9,10,11\}$, then there is a
$\bar{G}$-orbit in $\P^1$ of length at most $2$.
\end{lemma}

\begin{proof}
If $|\Omega|\in\{1,2,5,7,8,9,10,11\}$, then the group $\bar{G}$ is
either cyclic or dihedral, which implies that $\bar{G}$ has either
a fixed point or an orbit that consists of $2$ points,
respectively.
\end{proof}

\begin{lemma}
\label{lemma:S4-on-P2} Let $\bar{G}$ be a finite subgroup in
$\Aut(\P^2)$. Suppose that either $\bar{G}\cong\SS_4$, or
$\bar{G}\cong\A_4$. Then either  there is a $\bar{G}$-invariant
point in $\P^2$, or there exists at most one $\bar{G}$-orbit in
$\P^2$ of length~$3$.
\end{lemma}

\begin{proof}
If $\bar{G}\cong\SS_4$, then any point fixed by the normal
subgroup $\A_4\subset\bar{G}$ is also fixed by $\bar{G}$. Thus we
may assume that $\bar{G}\cong\A_4$ and there exists no
$\bar{G}$-invariant points in $\P^2$. Consider the unique subgroup
$\Gamma\subset\bar{G}\cong\A_4$ of index $3$. If there is a
$\bar{G}$-orbit $\Omega\subset\P^2$ such that $|\Omega|=3$, then
$\Gamma$ is a stabilizer of any point of $\Omega$. Let
$F_{\Gamma}\subset\P^2$ be the set of the points fixed by
$\Gamma$. If $F_{\Gamma}$ is infinite, then there is a line
$L\subset F_{\Gamma}$. Since the subgroup $\Gamma\subset\bar{G}$
is normal, the line $L$ is $\bar{G}$-invariant, and thus $\bar{G}$
fixes some point on $\P^2$ which is impossible by assumption. On
the other hand, if $\Gamma$ has a finite number $r$ of fixed
points, one has $r\le n+1=3$.
\end{proof}

\begin{theorem}
\label{theorem:Noether} Let $S$ be a rational surface with at most
Du Val singularities, and let $\sigma\colon \tilde{S}\to S$ be the
minimal resolution of singularities of the surface $S$. Then
$$
\rk\Pic\big(\bar{S}\big)+K_{\bar{S}}^2=\rk\Pic\big(S\big)+K_S^2+\sum_{P\in S}\mu\big(P\big)=10,%
$$
where $\mu(P)$ is the Milnor number\footnote{Recall that
$\mu(P)=0$ if $P\not\in\mathrm{Sing}(S)$, and $\mu(P)=n$ if $(S\ni
P)$ is a singularity of type $\mathbb{A}_n$, or $\mathbb{D}_{n}$,
or $\mathbb{E}_{n}$.} of the point $P$.
\end{theorem}

\begin{proof}
The required equality follows from  the classical Noether formula.
\end{proof}

Now we are ready to prove Theorem~\ref{theorem:main-5}. Let $G$ be
a~finite subgroup in $\GL_{5}(\mathbb{C})$ that does not contain
reflections, and let
$\phi\colon\GL_{5}(\mathbb{C})\to\PGL_{5}(\mathbb{C})$ be
the~natural projection. Put $\bar{G}=\phi(G)$. Let us identify
$\PGL_{5}(\mathbb{C})$ with $\mathrm{Aut}(\mathbb{P}^{4})$.

\begin{theorem}
\label{theorem:5-dim-1}  Suppose that $G$ is transitive and does
not have semi-invariants of degree at most~$4$. If $\C^{5}\slash
G$ is not weakly-exceptional, then there exists a
$\bar{G}$-invariant irreducible non-degenerate projectively normal
Fano type surface $S$ of degree $6$ with at most quotient
singularities that is not contained in a quadric hypersurface in
$\P^4$ such that its generic hyperplane section is a projectively
normal smooth curve of genus $3$.
\end{theorem}

\begin{proof}
Suppose that  $\C^{5}\slash G$ is not weakly-exceptional. Then
it follows from
Theorem~\ref{theorem:weakly-exceptional-quotient-criterion} that
there is an irreducible $\bar{G}$-invariant projectively normal
Fano type subvariety $S\subset\P^4$~such that
$$
\mathrm{deg}\big(S\big)\leqslant {4\choose \mathrm{dim}\big(S\big)},%
$$
and if $\mathrm{dim}(S)\geqslant 2$, then its general hyperplane
section is also projectively normal.

Since $G$ is transitive and $S$ is $\bar{G}$-invariant, the
subvariety $S$ is not contained in a~hyperplane in
$\mathbb{P}^{4}$. In particular, we see that $S$ is not a point.
Since $G$ does not have semi-invariants of degree at most $4$, we
see that $S$ is not a threefold.

If $S$ is a curve, then $\deg(S)\le 4$, and thus
$S$ is a smooth rational curve of degree
$4$, because $S$ is not contained in a~hyperplane in
$\mathbb{P}^{4}$. On the other hand, it follows from
\cite[Exercise~8.7]{Harris} that the secant variety of a smooth
rational curve of degree $4$ is a cubic hypersurface in $\P^4$.
This is impossible since $G$ does not have semi-invariants
of degree at most $4$ and
$S$ is $\bar{G}$-invariant. The obtained contradiction shows that $S$
cannot be a curve.

Thus, the subvariety $S$ is a surface. Then $S$ has at most
quotient singularities by \cite[Theorem~3.6]{Ko97}. Let $H$ be
a~hyperplane section of the~surface $S\subset\mathbb{P}^{4}$. Then
$H$ is projectively normal. Let $\mathcal{I}_{S}$ be the~ideal
sheaf of the surface $S$. Since $S$ is projectively normal, it
follows from the~Riemann--Roch theorem\footnote{
Recall that the Riemann--Roch theorem in the usual form
is valid for Cartier divisors on surfaces with quotient
(and even rational) singularities, see~\cite[Theorem~9.1]{YPG}.}
that
\begin{equation}
\label{equation:RR-2} {n+4\choose
n}-h^{0}\Big(\mathcal{O}_{\mathbb{P}^{4}}\big(n\big)\otimes\mathcal{I}_{S}\Big)=
h^{0}\Big(\mathcal{O}_{S}\big(nH\big)\Big)=1+\frac{n^{2}}{2}\Big(H\cdot
H\Big)-\frac{n}{2}\Big(H\cdot K_{S}\Big)%
\end{equation}
for any $n\geqslant 1$.  In particular, since $S$ is not contained
in a~hyperplane in $\mathbb{P}^{4}$, it follows from
$(\ref{equation:RR-2})$ applied for $n=1$ that $H\cdot H-H\cdot
K_{S}=8$. On the other hand, we know that $H\cdot H\geqslant 3$,
since $S$ is not contained in a~hyperplane in $\mathbb{P}^{4}$.
Moreover, since $S$ is a Fano type surface, the divisor $-K_{S}$
is big (see Definition~\ref{definition:Fano-type}), which implies
that $-H\cdot K_{S}>0$. Thus, we see that $H\cdot
H\in\{3,4,5,6,7\}$. Now plugging $n=2$ into
$(\ref{equation:RR-2})$, we see that
$6-h^{0}(\mathcal{O}_{\mathbb{P}^{4}}(2)\otimes\mathcal{I}_{S})=H\cdot
H$. In particular, we must have $H\cdot H\ne 5$, since $G$ does
not have semi-invariants of degree $2$. Thus, we see that  $H\cdot
H\in\{3,4,6\}$.

If $H\cdot H=3$, then either $S$ is
a~cone over a~smooth rational cubic curve, or
$S$ is a~smooth cubic scroll (see e.\,g. \cite{EisenbudHarris}).
If $S$ is a~cone over a~smooth
rational cubic curve, then its vertex is $\bar{G}$-invariant,
which is impossible since $G$ is transitive. If $S$ is a~smooth
cubic scroll, then there is a~unique line $L\subset S$ such that
$L^{2}=-1$, which implies that $L$ must be $\bar{G}$-invariant,
which is again impossible, because $G$ is transitive. Thus, we see
that $H\cdot H\ne 3$.

If $H\cdot H=6$, then $H\cdot K_{S}=-2$, which implies that $H$ is
a smooth curve of genus $3$ by the adjunction formula. Thus, if
$H\cdot H=6$, then we are done, because $H\cdot H=6$ implies that
$h^{0}(\mathcal{O}_{\mathbb{P}^{4}}(2)\otimes\mathcal{I}_{S})=0$.
Therefore, we may suppose that $H\cdot H=4$. Since
$6-h^{0}(\mathcal{O}_{\mathbb{P}^{4}}(2)\otimes\mathcal{I}_{S})=H\cdot
H$ and $S$ is not contained in a~hyperplane in $\mathbb{P}^{4}$,
there are two irreducible quadrics in $\P^4$ that contain~$S$. Let
us denote them by $Q$ and $Q^{\prime}$. Then $S=Q\cap Q^{\prime}$,
since $H\cdot H=4$.

If $S$ is singular, then  $|\mathrm{Sing}(S)|\leqslant 4$, because
$S$ has canonical singularities, since~$S$ is a~complete
intersection that has Kawamata log terminal singularities. But
$\mathrm{Sing}(S)$ is $\bar{G}$-invariant, which is impossible,
since $G$ is transitive. Thus, the surface $S$ is smooth.

Let $\mathcal{P}$ be a pencil that is generated by the quadrics
$Q$ and~$Q^{\prime}$. Then $\mathcal{P}$ contains exactly $5$
singular quadrics, which are simple quadric cones. Now it follows
from Lemma~\ref{lemma:small-orbit-on-P1} that there exists two
quadrics in $\mathcal{P}$, say $Q_{1}$ and $Q_{2}$, such that
$Q_{1}+Q_{2}$ is $\bar{G}$-invariant. The latter is impossible,
because $G$ does not have semi-invariants of degree at most $4$.
\end{proof}

To complete the proof of Theorem~\ref{theorem:main-5}
we are going to show that there are no surfaces~$S$ satisfying
the conditions implied by Theorem~\ref{theorem:5-dim-1}.
If $\C^5/G$ is weakly-exceptional, then $G$  is transitive
by~\cite[Proposition~2.1]{Pr00} and does not have semi-invariants
of degree at most $4$ by
Theorem~\ref{theorem:weakly-exceptional-quotient}. Thus, since
Fano type surfaces are rational (cf. \cite{Zh06}) and have big
anticanonical divisors by Definition~\ref{definition:Fano-type},
the assertion of Theorem~\ref{theorem:main-5} follows from
Theorem~\ref{theorem:5-dim-1} and the following

\begin{theorem}\label{theorem:no-Bordiga}
Let $S$ be an irreducible $\bar{G}$-invariant normal surface in
$\P^4$ of degree $6$, let $H$ be its general hyperplane section.
Then  $-K_S$ is not big if the following conditions are satisfied:
\begin{itemize}
\item[$(\mathrm{A})$] the surface $S$ is projectively normal,%
\item[$(\mathrm{B})$] the surface $S$ is rational,%
\item[$(\mathrm{C})$] the surface $S$ has at most quotient singularities,%
\item[$(\mathrm{D})$] the surface $S$ is not contained in a quadric hypersurface,%
\item[$(\mathrm{E})$] the curve $H$ is a smooth curve of genus $3$,%
\item[$(\mathrm{F})$] the curve $H$ is projectively normal (considered as a subvariety in $\P^3$),%
\item[$(\mathrm{G})$] the group $G$ is transitive.%
\end{itemize}
\end{theorem}

It should be pointed out that some of the conditions
$(\mathrm{A})$, $(\mathrm{B})$, $(\mathrm{C})$, $(\mathrm{D})$,
$(\mathrm{E})$ and $(\mathrm{F})$ in
Theorem~\ref{theorem:no-Bordiga} may be redundant. For example,
one can show that $(\mathrm{F})$ follows from $(\mathrm{A})$,
$(\mathrm{C})$, $(\mathrm{D})$ and $(\mathrm{E})$ (see the proof
of \cite[Theorem~2.1]{Homma80}, \cite[Theorem~4.4]{Mezzetti94} and
\cite{AlzatiBertoliniBesana97}). Still we prefer to keep them all,
because we do not care. In the remaining part of this section we
prove Theorem~\ref{theorem:no-Bordiga}. To do this we will need
two auxiliary results that may be useful on their own (cf.
\cite{Testa}).

\begin{lemma}\label{lemma:10-points-on-P2-easy}
Let $P_1, \ldots, P_{10}$ be  $10$ different points in $\P^2$, let
$\pi\colon\tilde{S}\to\P^2$ be the blow-up of the points
$P_1,\ldots, P_{10}$, let $\psi\colon\hat{S}\to\P^2$ be the
blow-up of the points $P_1,\ldots, P_{8}$, let $\hat{P}_{9}$ and
$\hat{P}_{10}$ be the preimages of the points $P_{9}$ and $P_{10}$
on the surface $\hat{S}$. Suppose that the linear system
$|-K_{\tilde{S}}|$ is empty,$-K_{\hat{S}}$ is nef and big, and
neither $\hat{P}_{9}$ nor $\hat{P}_{10}$ is contained in a curve
in $\hat{S}$ that has trivial intersection with $-K_{\hat{S}}$.
Then $-K_{\tilde{S}}$ is not big.
\end{lemma}

\begin{proof}
The surface $\hat{S}$ is a weak del Pezzo surface, i.e.
$-K_{\hat{S}}$ is big and nef. Let $\gamma\colon\hat{S}\to\bar{S}$
be its anticanonical model, i.\,e. $\gamma$  is given by the linear
system $|-nK_{\hat{S}}|$ for $n\gg 0$. Then $\bar{S}$ is a del
Pezzo surface with canonical singularities such that
$K_{\bar{S}}^2=1$, and $\gamma$ contracts all curves that have
trivial intersection with $-K_{\hat{S}}$ (they are $(-2)$-curves,
i.\,e. smooth rational curves with self-intersection $-2$). It is
well-known that $|-2K_{\bar{S}}|$ is free from base points and
induces a double cover $\delta\colon\bar{S}\to Q$, where $Q$ is a
quadric cone in $\mathbb{P}^{3}$.

By assumption both points $\gamma(\hat{P}_{9})$ and
$\gamma(\hat{P}_{10})$ are contained in a smooth locus of the
surface~$\bar{S}$. Moreover, there exists no curve in
$|-K_{\bar{S}}|$ that contains both points $\gamma(\hat{P}_{9})$
and $\gamma(\hat{P}_{10})$, because otherwise the linear system
$|-K_{\tilde{S}}|$ would not be empty. Thus, the points
$\delta\circ\gamma(\hat{P}_{9})$ and
$\delta\circ\gamma(\hat{P}_{10})$ are not contained in one ruling
of the quadric cone $Q$.

Let $\eta\colon \tilde{S}\to\hat{S}$ be the blow up of the points
$\hat{P}_{9}$ and $\hat{P}_{10}$, and let $E_{9}$ and $E_{10}$ be
the exceptional divisors of the blow up $\eta$ such that
$\eta(E_{9})=\hat{P}_{9}$ and $\eta(E_{10})=\hat{P}_{10}$. Then we
have a diagram
$$
\xymatrix{ \tilde{S}\ar@{->}[r]^{\eta}\ar@{->}[d]_{\pi}&
\hat{S}\ar@{->}[ld]^{\psi}\ar[rd]^{\gamma}&&\\
\P^2&&\bar{S}\ar@{->}[r]^{\delta}&Q.}
$$

Since
$-2K_{\tilde{S}}+E_{9}+E_{10}\sim\eta^{*}\big(-2K_{\hat{S}}\big)-E_{9}-E_{10}$,
the linear system $|-2K_{\tilde{S}}+E_{9}+E_{10}|$ does not have
base curves except possibly the curves $E_{9}$ and $E_{10}$,
because the points $\delta\circ\gamma(\hat{P}_{9})$ and
$\delta\circ\gamma(\hat{P}_{10})$ are not contained in one ruling
of the quadric cone $Q$. Indeed, the base locus of the linear
system $\eta(|-2K_{\tilde{S}}+E_{9}+E_{10}|)$ consists of the
points of the surface $\hat{S}$  that are mapped to one of the
points $\delta\circ\gamma(\hat{P}_{9})$ or
$\delta\circ\gamma(\hat{P}_{10})$ by $\delta\circ\gamma$, because
$\gamma$ is an isomorphism in a neighborhood of the points
$\hat{P}_{9}$ and $\hat{P}_{10}$, and
$\delta\circ\gamma(\hat{P}_{9})$ and
$\delta\circ\gamma(\hat{P}_{10})$ are not contained in one ruling
of the quadric cone $Q$. Thus, the divisor
$-2K_{\tilde{S}}+E_{9}+E_{10}$ is nef and big, because
$(-2K_{\tilde{S}}+E_{9}+E_{10})^2=2$ and
$$(-2K_{\tilde{S}}+E_{9}+E_{10})\cdot E_{9}=
(-2K_{\tilde{S}}+E_{9}+E_{10})\cdot E_{10}=1.$$
One the other hand,
we see that $-K_{\tilde{S}}\cdot(-2K_{\tilde{S}}+E_{9}+E_{10})=0$,
which implies that $-K_{\tilde{S}}$ cannot be big.
\end{proof}

\begin{lemma}\label{lemma:10-points-on-P2}
Let a group $\bar{G}$ act faithfully on the projective plane
$\P^2$. Let $\PP=\{P_1, \ldots, P_{10}\}$ be a $\bar{G}$-invariant
collection of $10$ points on $\P^2$, and denote by $\pi\colon
\tilde{S}\to\P^2$ the blow-up of the points $P_1, \ldots, P_{10}$.
Suppose that the following conditions hold:
\begin{itemize}
\item[$(\mathrm{i})$] there are no cubic curves passing through
all points of $\PP$; \item[$(\mathrm{ii})$] if a
$\bar{G}$-invariant curve $C\subset\P^2$ passes through at least
$r$ of the points of $\PP$, then one has~\mbox{$4\deg(C)-r\ge 4$};
\item[$(\mathrm{iii})$] there are no $\bar{G}$-orbits of length at
most $2$ on $\P^2$; \item[$(\mathrm{iv})$] there are no
$\bar{G}$-orbits of length at most $4$ contained in
$\P^2\setminus\PP$;
\end{itemize}
Then the anticanonical class $-K_{\tilde{S}}$ is not big.
\end{lemma}

\begin{proof}
To start with, we are going to prove that the group $\bar{G}$ acts
transitively on the set $\PP$. Keeping in mind
 $(\mathrm{iii})$, one can easily see that for a
non-transitive action there are the following possibilities to
split $\PP$ into $\bar{G}$-orbits: either $\PP$ is a union of two
$\bar{G}$-orbits of length $5$, or it is a union of a
$\bar{G}$-orbit of length $6$ and a $\bar{G}$-orbit of length $4$,
or it is a union of a $\bar{G}$-orbit of length~$4$ and two
$\bar{G}$-orbits of length~$3$, or it is a union of a
$\bar{G}$-orbit of length~$7$ and a $\bar{G}$-orbit of length $3$.
Let us exclude these possibilities case by case.

Assume that $\PP$ contains a $\bar{G}$-orbit $\PP'$ of length $5$.
If there are at least $3$ points of $\PP'$ that lie on a line $L$,
then each of the images of $L$ under the action of $\bar{G}$ also
contains at least $3$ points of $\PP'$. This is possible only if
the $\bar{G}$-orbit of $L$ consists of at most $2$ lines. The
latter implies that there is a $\bar{G}$-orbit on $\P^2$ that
consists of at most $2$ points, which contradicts
 $(\mathrm{iii})$. Thus, the $5$ points of $\PP'$ are in
general position, so that there is a unique conic $C$ in $\P^2$
than passing through all points of $\PP'$, and $C$ is
non-singular. By Lemma~\ref{lemma:small-orbit-on-P1} there is a
$\bar{G}$-orbit on $C$ of length at most $2$, which contradicts
 $(\mathrm{iii})$.

Assume that $\PP$ contains a $\bar{G}$-orbit $\PP'$ of length $4$.
The same argument as above shows that the points of $\PP'$ are in
general position, so that the pencil of conics passing through the
points of $\PP'$ has exactly $3$ degenerate members, and each of
them is a union of two distinct lines. The intersection points of
these pairs of lines gives a $\bar{G}$-orbit $\PP''$ of length at
most $3$. By  $(\mathrm{iv})$ one has $\PP''\subset\PP$, so that
by
 $(\mathrm{iii})$ one has $|\PP''|=3$. Thus,
$\PP'''=\PP\setminus (\PP'\cup\PP'')$ is another $\bar{G}$-orbit
of length $3$ contained in $\PP$ by  $(\mathrm{iii})$.

Note that $\bar{G}$ acts faithfully on the finite set $\PP'$ since
an automrphism of $\P^2$ is defined by the images of $4$ points in
general position. Hence there is an embedding
$\bar{G}\hookrightarrow\SS_4$. Keeping in mind that $|\bar{G}|$ is
divisible by $12$ since $\bar{G}$ has orbits of lengths $3$ and
$4$, we see that either $\bar{G}\cong\SS_4$ or $\bar{G}\cong\A_4$.
Lemma~\ref{lemma:S4-on-P2} implies that there is a
$\bar{G}$-invariant point on $\P^2$, which is impossible by
 $(\mathrm{iii})$.

Assume that $\PP$ contains a $\bar{G}$-orbit $\PP'$ of length $7$.
We are going to prove that the points of~$\PP'$ are in general
position, i.\,e. there are no lines passing through $3$ points of
$\PP'$ and there are no conics passing through $6$ points of
$\PP'$.

Suppose that there is a line $L_1$ passing through $3$ points of
$\PP'$. Let $\LL=\{L_1, \ldots, L_k\}$ be the $\bar{G}$-orbit of
the line $L$. Denote by $l$ the number of lines from $\LL$ passing
through each point from $\PP'$, and by $p$ the number of points
from $\PP'$ lying on each line from $\LL$. One has $p\ge 3$ by
assumption and $p\le 4$ by  $(\mathrm{i})$. On the other hand, one
has $l\le 3$ since $|\PP'|=7$. Using the equality $7l=pk$, one
finds that $k$ is divisible by $7$, and thus $l=p=3$ and $k=7$. It
is easy to see that there are no irreducible conics passing
through $6$ of the points from $\PP'$ (indeed, otherwise there
would exist a line from $\LL$ intersecting such conic in at least
$3$ points). Let $\tilde{\Sigma}$ be a surface obtained by blowing
up the $7$ points of $\PP'$ and $\Sigma$ be an anticanonical model
of $\tilde{\Sigma}$. Then $\Sigma$ is a del Pezzo surface of
degree $2$ with $7$ Du Val singular points of type $\mathbb{A}_1$
(that are images of the proper transforms of lines from $\LL$),
which easily leads to a contradiction.

Suppose that there is an irreducible conic $C$ passing through at
least $6$ points of $\PP'$. Then~$C$ is $\bar{G}$-invariant, since
otherwise there would exist another irreducible conic passing
through at least $6$ points of $\PP'$ and thus intersecting $C$ in
at least $5$ points. On the other hand, the curve~$C$ cannot be
$\bar{G}$-invariant by  $(\mathrm{ii})$.

We see that the $7$ points of $\PP'$ are in general position, so
that the surface $\tilde{\Sigma}$ obtained by the blow-up
$\pi'\colon\tilde{\Sigma}\to\P^2$ of the points of $\PP'$ is a
smooth del Pezzo surface of degree $2$. Moreover,
$\bar{G}\subset\Aut(\tilde{\Sigma})$, so that
$|\Aut(\tilde{\Sigma})|$ is divisible by $7$.
Using~\cite[Table~6]{DoIs06}, one obtains that $\bar{G}$ is a
subgroup of the Klein group $\PSL_2(\F_7)$; moreover, $\bar{G}$
must be isomorphic either to $\PSL_2(\F_7)$ itself, or to a group
$F_{21}\cong\mathbb{Z}_{7}\rtimes\mathbb{Z}_{3}$, or to the cyclic
group $\Z_7$. The former two cases are impossible, since in
$\tilde{\Sigma}$ is $\bar{G}$-minimal if $\bar{G}$ contains
$F_{21}$ by~\cite[Theorem~6.17]{DoIs06}, which contradicts the
existence of the $\bar{G}$-invariant morphism $\pi'$. The latter
case is impossible since $|\bar{G}|$ must be also divisible by $3$
since $\bar{G}$ has an orbit of length $3$ on $\P^2$.

We see that $\bar{G}$ acts transitively on the points of $\PP$.
Let us consider the following condition~$(\mathrm{*})$: there
exist $8$ points of $\PP$ (say, $P_1, \ldots, P_8$) such that the
surface $\tilde{\Sigma}$ obtained by blowing up these points is a
weak del Pezzo surface, and the preimages of the other two points
of $\PP$ (i.\,e., $P_9$~and~$P_{10}$) are not contained in
$(-2)$-curves of $\tilde{\Sigma}$. Since $|-K_{\tilde{S}}|$ is
empty, because there are no cubic curves passing through all
points of $\PP$, we see that it follows from
Lemma~\ref{lemma:10-points-on-P2-easy} that $-K_{\tilde{S}}$ is
not big assuming that  $(\mathrm{*})$ holds. Let us show that the
failure of~$(\mathrm{*})$ leads  to a contradiction.

Suppose that  $(\mathrm{*})$ does not hold. It means that either
there exists a line containing at least~$4$ points of $\PP$, or
there exists an irreducible conic containing at least $7$ points
of $\PP$, or there exists an irreducible cubic containing at least
$9$ points of $\PP$ and singular at one of these points.
Moreover,~$(\mathrm{i})$ implies that in the latter list of
possibilities one can replace ``at least'' by ``exactly''.

Suppose that there exists an irreducible cubic, say, $Z_{1}$,
containing $9$ points of $\PP$, say, $P_{1}, \ldots, P_{9}$, and
singular at one of these points, say, $P_{1}$. Consider an element
$g_2\in\bar{G}$ such that $g_2(P_1)=P_2$ (it exists since the
action of $\bar{G}$ on $\PP$ is transitive) and put
$Z_2=g_2(Z_1)$. Note that $Z_2\neq Z_1$ since $Z_1$ is smooth at
$P_2$ while $Z_2$ is singular at that point. If $P_1\in Z_2$, then
\begin{equation}
\label{eq:9-10} 9=Z_1\cdot Z_2\ge\sum_{i=1}^{2}\mult_{P_i}(Z_1)\mult_{P_i}(Z_2)+\sum\limits_{i=3}^9\mult_{P_i}(Z_1)\mult_{P_i}(Z_2)=10,%
\end{equation}
which is a contradiction. If $P_1\not\in Z_2$, consider an element
$g_3\in\bar{G}$ such that $g_3(P_1)=P_3$ and put $Z_3=g_3(Z_1)$.
If $P_1\in Z_3$, replace $Z_2$ by $Z_3$. If $P_1\not\in Z_3$, note
that $P_2\in Z_3$ and $P_3\in Z_2$ and replace $Z_1$ by $Z_3$. In
both cases a computation similar to~(\ref{eq:9-10}) leads to a
contradiction.

Suppose that there exists an irreducible conic $C_1$ passing
through $7$ points of $\PP$. Let $\CC=\{C_1, \ldots, C_k\}$ be the
$\bar{G}$-orbit of the conic $C_1$. Denote by $c$ the number of
conics from $\CC$ passing through each point from $\PP$; note that
there are exactly $7$ points from $\PP$ lying on each conic from
$\CC$. Thus one has $7k=10c$. Put $\PP_i=\PP\setminus C_i$, $1\le
i\le k$. If $k\ge 4$, then there are indices $i_1, i_2\in\{1,
\ldots, k\}$ such that $\PP_{i_1}\cap\PP_{i_2}\neq\varnothing$.
Hence the conics $C_{i_1}$ and $C_{i_2}$ intersect by at least $5$
points which is impossible. Therefore $k\le 3$. Moreover, the
cases $k=1$ and $k=2$ are impossible by  $(\mathrm{ii})$, so that
$k=3$ and $10c=21$, which is again a contradiction.

We conclude that there exists a line $L_1$ passing through $4$
points of $\PP$. Let $\LL=\{L_1,\ldots, L_k\}$ be the
$\bar{G}$-orbit of the line $L_1$. Denote by $l$ the number of
lines from $\LL$ passing through each point from $\PP$; note that
there are exactly $4$ points from $\PP$ lying on each line from
$\LL$. One has $l\le 3$ since $|\PP|=10$ and $k\ge 5$ by
 $(\mathrm{ii})$. Using the equality $10l=4k$, one finds
that $l=2$ and $k=5$. Thus, the set $\LL$ consists of $5$ lines in
general position. By  $(\mathrm{ii})$, the group $\bar{G}$ acts
transitively on the set $\LL$. Hence there exists a
$\bar{G}$-orbit $\Omega\subset\P^2$ that consists of $5$ points in
general position, so that there is a unique (and thus
$\bar{G}$-invariant) smooth conic $C$ passing through the points
of $\Omega$. By Lemma~\ref{lemma:small-orbit-on-P1}, there exists
a $\bar{G}$-orbit $\Omega'\subset C$ of length at most~$2$. A
contradiction with  $(\mathrm{iii})$ completes the proof.
\end{proof}

Now we are ready to prove Theorem~\ref{theorem:no-Bordiga}. Let
$S$ be an irreducible normal surface in $\P^4$ of degree~$6$, let
$H$ be its general hyperplane section, and let $\bar{G}$ be a
finite subgroup in $\mathrm{Aut}(\P^5)$ such that $S$ is
$\bar{G}$-invariant. As usual, we identify $\mathrm{Aut}(\P^5)$
with $\PGL_5(\C)$. Suppose that the conditions $(\mathrm{A})$,
$(\mathrm{B})$, $(\mathrm{C})$, $(\mathrm{D})$, $(\mathrm{E})$ and
$(\mathrm{F})$ in Theorem~\ref{theorem:no-Bordiga} are satisfied.

\begin{lemma}\label{lemma:Bordiga-hyperelliptic}
The curve $H$ is not hyperelliptic.
\end{lemma}

\begin{proof}
The required assertion follows from  $(\mathrm{F})$ (see
\cite[Theorem~2.1]{Homma80}).
\end{proof}

Let $f\colon\tilde{S}\to S$ be the minimal resolution of
singularities of the surface $S$, and let $\tilde{H}$ be the
proper transform of the curve $H$ on the surface $\tilde{S}$. Then
the actions of the group $\bar{G}$ lifts to the surface
$\tilde{S}$, and  $\tilde{H}\sim f^{*}(H)$, because $H$ is a
general hyperplane section of the surface $S$. Note that
$-K_{S}\cdot H=-K_{\tilde{S}}\cdot\tilde{H}=2$ by the adjunction
formula and $(\mathrm{E})$.

\begin{lemma}\label{lemma:Bordiga-anticanclass-empty}
The linear systems $|-K_{\tilde{S}}|$ and $|-K_{S}|$ are empty.
\end{lemma}

\begin{proof}
If $|-K_{S}|$ contains a $\bar{G}$-invariant fixed curve $C$, then
$C$ must be either a line or a conic in $\mathbb{P}^{4}$, because
$H\cdot C\leqslant -K_{S}\cdot H=2$. Hence, it follows from
$(\mathrm{G})$ that $|-K_{S}|$ does not contain fixed curves if
$|-K_{S}|$ is not empty.

Suppose that $|-K_{S}|\neq\varnothing$. Then every curve in $|-K_{S}|$ is
either a union of two lines or a conic in $\mathbb{P}^{4}$,
because $-K_{S}\cdot H=2$. Hence
either $|-K_{S}|$ is free from base points or its base locus must
consist of at most $4$ points. The latter case is impossible by
$(\mathrm{G})$. Hence, the linear system  $|-K_{S}|$ is free. In
particular, the divisor $-K_{S}$ is Cartier, which implies that
the surface $S$ has Du Val singularities by $(\mathrm{C})$. Since
$-K_{S}\cdot H=2$, a generic curve in $|-K_{S}|$ must be either a
disjoint union of two lines or a smooth conic, which is impossible
by the adjunction formula.

Finally, $|-K_{\tilde{S}}|=\varnothing$ is implied by $|-K_S|=\varnothing$.
\end{proof}

\begin{lemma}\label{lemma:Bordiga-to-P2}
The equalities $K_{\tilde{S}}^{2}=-1$ and
$h^{0}(\mathcal{O}_{\tilde{S}}(K_{\tilde{S}}+\tilde{H}))=3$ hold,
the linear system $|K_{\tilde{S}}+\tilde{H}|$ is free from base
points and induces a birational morphism
$\pi\colon\tilde{S}\to\mathbb{P}^{2}$.
\end{lemma}

\begin{proof}
It follows from the~Riemann--Roch theorem and the Nadel--Shokurov
vanishing theorem  that
$$
h^{0}\Big(\mathcal{O}_{\tilde{S}}\big(K_{\tilde{S}}+
\tilde{H}\big)\Big)=\chi\Big(\mathcal{O}_{\tilde{S}}\big(K_{\tilde{S}}+
\tilde{H}\big)\Big)=
\chi\big(\mathcal{O}_{\tilde{S}}\big)+
\frac{\tilde{H}\cdot(\tilde{H}-K_{\tilde{S}})}{2}=3,%
$$
because $\chi(\mathcal{O}_{\tilde{S}})=1$ by $(\mathrm{B})$. One
the other hand, there is an exact sequence of cohomology groups
$$
0\to H^{0}\Big(\mathcal{O}_{\tilde{S}}\big(K_{\tilde{S}}\big)\Big)\to H^{0}\Big(\mathcal{O}_{\tilde{S}}\big(K_{\tilde{S}}+\tilde{H}\big)\Big)\to H^{0}\Big(\mathcal{O}_{\tilde{H}}\big(K_{\tilde{S}}\big)\Big),%
$$
which implies that
$h^{0}(\mathcal{O}_{\tilde{S}}(K_{\tilde{S}}+\tilde{H}))=3$ and
$|K_{\tilde{S}}+\tilde{H}|$ does not have base points contained in
the curve $\tilde{H}$, because
$h^{0}(\mathcal{O}_{\tilde{S}}(K_{\tilde{S}}))=0$ by
$(\mathrm{B})$, $h^{0}(\mathcal{O}_{\tilde{H}}(K_{\tilde{H}}))=3$
by~$(\mathrm{E})$, and $|K_{\tilde{H}}|$ is free from base points
by  $(\mathrm{E})$. Thus, every base curve of the linear system
$|K_{\tilde{S}}+\tilde{H}|$ is $f$-exceptional. In particular, the
divisor $K_{\tilde{S}}+\tilde{H}$ is nef, since $K_{\tilde{S}}$ is
$f$-nef, because the resolution of singularities~$f$ is minimal.
Thus, we see that
\begin{equation}
\label{equation:Bordiga-K-square}
K_{\tilde{S}}^2+2=\Big(K_{\tilde{S}}+\tilde{H}\Big)^{2}\geqslant 0,%
\end{equation}
which gives $K_{\tilde{S}}^2\geqslant -2$. On the other hand, it
follows from the~Riemann--Roch theorem that
$$
h^{0}\Big(\mathcal{O}_{\tilde{S}}\big(-K_{\tilde{S}}\big)\Big)\geqslant\chi\Big(\mathcal{O}_{\tilde{S}}\big(-K_{\tilde{S}}\big)\Big)=\chi\big(\mathcal{O}_{\tilde{S}}\big)+K_{\tilde{S}}^2=1+K_{\tilde{S}}^2,%
$$
because
$h^{2}(\mathcal{O}_{\tilde{S}}(-K_{\tilde{S}}))=h^{0}(\mathcal{O}_{\tilde{S}}(2K_{\tilde{S}}))=0$
and $\chi(\mathcal{O}_{\tilde{S}})=1$ by $(\mathrm{B})$. Since
$|-K_{\tilde{S}}|$ is empty by
Lemma~\ref{lemma:Bordiga-anticanclass-empty}, we see that
$K_{\tilde{S}}^2\leqslant -1$. Thus, either $K_{\tilde{S}}^2=-1$
or $K_{\tilde{S}}^2=-2$. Applying Theorem~\ref{theorem:Noether},
we see that
\begin{equation}
\label{equation:Bordiga-Noether-formula}
\mathrm{rk}\Big(\mathrm{Pic}\big(\tilde{S}\big)\Big)=10-K_{\tilde{S}}^2\leqslant 12,%
\end{equation}
which implies that $|\mathrm{Sing}(S)|\leqslant 11$.  By
$(\mathrm{G})$,  we see that either $S$ is smooth, or
$5\leqslant|\mathrm{Sing}(S)|\leqslant 11$. Moreover, it follows
from $(\ref{equation:Bordiga-Noether-formula})$ and $(\mathrm{C})$
and $(\mathrm{G})$ that only one of the following cases is
possible:
\begin{itemize}
\item the surface $S$ is smooth and $f$ is an isomorphism,
\item $5\leqslant|\mathrm{Sing}(S)|\leqslant 11$, and $f$ contracts exactly one smooth irreducible rational curve to every singular point of the surface $S$,%
\item $|\mathrm{Sing}(S)|=5$, and $f$ contracts exactly two smooth irreducible rational curves that intersect transversally in one point to every singular point of the surface $S$.%
\end{itemize}

The linear system $|K_{\tilde{S}}+\tilde{H}|$ induces a rational
map $\pi\colon\tilde{S}\dasharrow\mathbb{P}^{2}$. If
$|K_{\tilde{S}}+\tilde{H}|$ is free from base points, then it
follows from $(\ref{equation:Bordiga-K-square})$ that either
$K_{\tilde{S}}^2=-1$ and $\pi$ is a birational morphism, or
$K_{\tilde{S}}^2=-2$ and $|K_{\tilde{S}}+\tilde{H}|$ is composed
of a pencil. Moreover, if $|K_{\tilde{S}}+\tilde{H}|$ is free
from base points and is composed of a pencil, then
$\pi(\tilde{S})$ must be a smooth conic and
$\pi\vert_{\tilde{H}}\colon\tilde{H}\to\pi(\tilde{S})$ must be a double
cover, because $(K_{\tilde{S}}+\tilde{H})\cdot\tilde{H}=4$. By
Lemma~\ref{lemma:Bordiga-hyperelliptic}, we know that $\tilde{H}$
is not hyperelliptic. Thus, the case when
$|K_{\tilde{S}}+\tilde{H}|$ is free from base points and
$K_{\tilde{S}}^2=-2$ is impossible. Therefore, to complete the
proof it is enough to prove that $|K_{\tilde{S}}+\tilde{H}|$ is
free from base points.

Suppose that $|K_{\tilde{S}}+\tilde{H}|$ has a base point
$P\in\tilde{S}$. Then it follows from \cite[Theorem~1]{Reider88}
that there exists an effective Cartier divisor $E$ on the surface
$\tilde{S}$ such that $P\in\mathrm{Supp}(E)$ and either
$\tilde{H}\cdot E=0$ and $E^{2}=-1$ or $\tilde{H}\cdot E=1$ and
$E^{2}=0$.

If $\tilde{H}\cdot E=0$ and $E^{2}=-1$, then $E$ is
$f$-exceptional, which implies that $S$ must be singular at the
point $f(P)$. If $\tilde{H}\cdot E=1$ and $E^{2}=0$ and $S$ is
smooth, then $E$ is a line in $S\subset\mathbb{P}^4$, which
implies that $-1=K_{\tilde{S}}\cdot E+1=(E+\tilde{H})\cdot
E\geqslant 0$ by the adjunction formula, because
$K_{\tilde{S}}+\tilde{H}$ is nef. Thus, in both cases the surface
$S$ must be singular.

Let $r$ be the number of $f$-exceptional irreducible curves. Let
us denote these curves by $E_{1},\ldots,E_{r}$. Put
$E=\tilde{L}+\sum_{i=1}^{r}a_{i}E_{i}$, where $\tilde{L}$ is an
effective Cartier divisor on the surface~$\tilde{S}$ such that
none of its components is $f$-exceptional, and $a_{i}$ is a
non-negative integer. Put $n_{i}=-E_{i}^{2}$. Then $n_{i}\geqslant
2$ for every $i\in\{1,\ldots,r\}$, because $f$ is minimal.

Suppose that $\tilde{H}\cdot E=0$ and $E^{2}=-1$. Then $E$ is
$f$-exceptional, which simply means that $\tilde{L}=0$. If $f$
contracts exactly one smooth irreducible rational curve to every
singular point of the surface $S$, then
$$
-2=E^2=\Big(\sum_{i=1}^{r}a_{i}E_{i}\Big)^2=-\sum_{i=1}^{r}a_{i}^2n_{i}\leqslant -2\sum_{i=1}^{r}a_{i}^2,%
$$
which is a contradiction. Thus, we see that
$|\mathrm{Sing}(S)|=5$, and $f$ contracts exactly two smooth
irreducible rational curves that intersect transversally in one
point to every singular point of the surface $S$. Without loss of
generality, we may assume that $f(E_{1})=f(E_{2})=f(P)$ and
$f(P)\ne f(E_{i})$ for every $i\not\in\{1,2\}$. Then
$$
-1=E^{2}=-a_{1}^2n_{1}+2a_{1}b_{2}-a_{2}^2n_{2}\leqslant -2a_{1}^2+2a_{1}a_{2}-2a_{2}^2=-a_{1}^2-a_{2}^2-\Big(a_{1}+a_{2}\Big)^2,%
$$
which immediately leads to a contradiction. Thus, we see that the
case when $\tilde{H}\cdot E=0$ and $E^{2}=-1$ is
impossible\footnote{As was pointed out to us by Yu.\,Prokhorov,
the contradiction can be obtained much easier in this case.
Namely, if $E^{2}=-1$ and $E$ is $f$-exceptional, then
$(K_{\tilde{S}}+E)\cdot E\geqslant 0$, since $K_{\tilde{S}}$ is
$f$-nef, which implies that $S$ has non-rational singularities by
\cite{Artin66}. However, the surface $S$ has rational
singularities by $(\mathrm{C})$.}.

Therefore, we see that $\tilde{H}\cdot E=1$ and $E^{2}=0$. Put
$L=f(\tilde{L})$. Then $L$ is a line in $S\subset\mathbb{P}^4$,
because $H\cdot L=\tilde{H}\cdot E=1$. We see that $\tilde{L}$ is
an irreducible smooth rational curve. Then it follows from the
adjunction formula that
$-1-\tilde{L}^2=(K_{\tilde{S}}+\tilde{H})\cdot \tilde{L}\geqslant
0$, because we already proved that the divisor
$K_{\tilde{S}}+\tilde{H}$ is nef. Thus, we see that $\tilde{L}$ is
a smooth rational curve on the surface~$\tilde{S}$ such that
$\tilde{L}^2\leqslant -1$. If $L\cap\mathrm{Sing}(S)=\varnothing$,
then
$0=E^2=\tilde{L}^2+(\sum_{i=1}^{r}a_{i}E_{i})^2\leqslant\tilde{L}^2$,
which is impossible, since $\tilde{L}^2\leqslant -1$. Thus, we see
that $L\cap\mathrm{Sing}(S)\ne\varnothing$.

If $f$ contracts exactly one smooth irreducible rational curve to
every singular point of the surface $S$, then
$$
0=E^2\leqslant\tilde{L}^2-\sum_{i=1}^{r}\Big(a_{i}^2n_{i}-2a_{i}\Big)\leqslant\tilde{L}^2,
$$
because either $\tilde{L}\cdot E_{i}=0$ or $\tilde{L}\cdot
E_{i}=1$, since the curve $L$ is smooth (it is a line in
$S\subset\P^4$). Keeping in mind that $\tilde{L}^2\leqslant -1$,
we see that $|\mathrm{Sing}(S)|=5$, and $f$ contracts exactly two
smooth irreducible rational curves that intersect transversally in
one point to every singular point of the surface $S$. So, we have
$r=10$.

Let us denote five singular points of the surface $S$ by $O_{1}$,
$O_{2}$, $O_{3}$, $O_{4}$ and $O_{5}$. Without loss of generality,
we may assume that $f(E_{i})=f(E_{i+1})=O_{\lceil i/2\rceil}$ for
every $i\in\{1,\ldots,9\}$. Since $L$ is smooth and $(\mathrm{G})$
holds, we may assume that $\tilde{L}\cdot E_{1}=\tilde{L}\cdot
E_{3}=\tilde{L}\cdot E_{5}=\tilde{L}\cdot E_{7}=\tilde{L}\cdot
E_{9}=1$ and~$\tilde{L}$ does not intersect the curves $E_{2}$,
$E_{4}$, $E_{6}$, $E_{8}$ and $E_{10}$. Then
$$
0=E^2=\tilde{L}^2-\sum_{i=1}^{10}a_{i}^2n_{i}+2\sum_{i=0}^{4}a_{2i+1}+2\sum_{i=0}^{4}a_{2i+1}a_{2i+2},
$$
which implies that $\tilde{L}^2\geqslant 0$. Indeed, for every
$i\in\{1,3,5,7,9\}$, we have
$$
2a_{i}+2a_{i}a_{i+1}-a_{i}^2n_{i}-a_{i+1}^2n_{i+1}\leqslant-\Big(a_{i}-a_{i+1}\Big)^2-\Big(a_{i}-1\Big)^2-(a_{i+1}^2-1)\leqslant 0,%
$$
because $a_{i}$ and $a_{i+1}$ are non-negative integers. Since we
already proved that $\tilde{L}^2\leqslant -1$, we see that our
assumption that $|K_{\tilde{S}}+\tilde{H}|$ has a base point was
wrong, which completes the proof.
\end{proof}

Note that the birational morphism
$\pi\colon\tilde{S}\to\mathbb{P}^{2}$ is $\bar{G}$-equivariant,
because the line bundle $K_{\tilde{S}}+\tilde{H}$ is
$\bar{G}$-invariant. Since $K_{\tilde{S}}^2=-1$, the morphism
$\pi$ contracts $10$ irreducible smooth rational curves. Let us
denote them by $E_{1}$, $E_{2}$, $E_{3}$, $E_{4}$, $E_{5}$,
$E_{6}$, $E_{7}$, $E_{8}$, $E_{9}$ and $E_{10}$.

\begin{lemma}\label{lemma:Bordiga-10-points}
The curves $E_{1}$, $E_{2}$, $E_{3}$, $E_{4}$, $E_{5}$, $E_{6}$,
$E_{7}$, $E_{8}$, $E_{9}$ and $E_{10}$ are pairwise disjoint.
\end{lemma}

\begin{proof}
To complete the proof, we must show that $E_{i}^2=-1$ for every
$i\in\{1,\ldots,10\}$. Suppose that this
 is not true. Then there should be at least
one curve among $E_{1},\ldots,E_{10}$ whose self-intersection is
not $-1$. On the other hand, there should be at least one curve
among $E_{1},\ldots,E_{10}$ whose self-intersection is $-1$. We
may assume that there is $k\in\{2,\ldots,9\}$ such that
$E_{i}^{2}=-1$ for every $i\leqslant k$, and $E_{i}^{2}\ne -1$ for
every $i>k$. By the adjunction formula, we have
$K_{\tilde{S}}\cdot E_{i}=-2-E_{i}^2\geqslant -1$. Since
$(K_{\tilde{S}}+\tilde{H})\cdot E_{i}=0$, we see that
$E^{2}_{i}\ne -1$ if and only if $E_{i}^2=-2$ and $E_{i}$ is
contracted by $f$ to a singular point of the surface $S$. On the
other hand, if $E_{i}^{2}=-1$, then $\tilde{H}\cdot E_{i}=1$,
which implies that $f(E_{i})$  is a line in $S\subset\P^4$. In
particular, we see that $E_{i}^{2}=-2$ for every $i>k$. Since the
set $\cup_{i=1}^{10}E_{i}$ is $\bar{G}$-invariant, it easily
follows from $(\mathrm{G})$ that $k=5$. Moreover, it follows from
$(\mathrm{G})$ that $\bar{G}$ acts transitively on the set
$\{E_{6}, E_{7}, E_{8}, E_{9}, E_{10}\}$.

The birational morphism $\pi$ contracts the curves
$E_{1},\ldots,E_{10}$ to $5$ points in $\P^2$. Let us denote these
points by $O_{1}$,  $O_{2}$,  $O_{3}$,  $O_{4}$ and  $O_{5}$.
Without loss of generality, we may assume that $E_{i}$
and~$E_{i+5}$ are contracted by $f$ to the point $O_{i}$ for every
$i\in\{1,\ldots,5\}$.

Let $L$ be a line in $\P^2$, and let $\tilde{L}$ be its proper
transform on the surface $\tilde{S}$. Then it follows from the
adjunction formula that $\tilde{H}\cdot\tilde{L}=3+\tilde{L}^{2}$,
which implies that $\tilde{L}^{2}\geqslant -3$, and
$\tilde{L}^{2}=-3$ if and only if $\tilde{L}$ is contracted by $f$
to a singular point of the surface $S$. In particular, we see that
there exists no line in $\P^2$ that passes through all points
$O_{1},\ldots,O_{5}$. Similarly, if $L$ contains four points among
$O_{1},\ldots,O_{5}$, then $L$ must be $\bar{G}$-invariant and
$\tilde{L}^{2}\leqslant -3$, which implies that $f(\tilde{L})$ is
a $\bar{G}$-invariant point in $S$, which is impossible by
$(\mathrm{G})$.

We see that no $4$ among the points $O_{1}$, $O_{2}$, $O_{3}$,
$O_{4}$ and $O_{5}$ lie on a line in $\P^2$. In particular, there
exists unique reduced conic in $\P^2$ that passes through the
points $O_{1}$, $O_{2}$, $O_{3}$, $O_{4}$ and $O_{5}$. Let us
denote this conic by $C$, and let us denote by $\tilde{C}$ its
proper transform on the surface $\tilde{S}$. Then $C$ and
$\tilde{C}$ are $\bar{G}$-invariant curves. If $C$ is irreducible,
then $\tilde{C}^{2}\leqslant -1$ and it follows from the
adjunction formula that
$\tilde{H}\cdot\tilde{C}=4+\tilde{C}^{2}\leqslant 3$, which
implies that $f(\tilde{C})$ is contained in a proper
$\bar{G}$-invariant linear subspace in $\P^4$, which is impossible
by  $(\mathrm{G})$. Thus, the conic $C$ is reducible. Then at
least one component of the curve~$f(\tilde{C})$ (not necessary the
linear span of this component) is contained in a proper
$\bar{G}$-invariant linear subspace in $\P^4$, which is impossible
by  $(\mathrm{G})$.
\end{proof}

Put $P_{i}=\pi(E_{i})$, $1\le i\le 10$, and put
$\PP=\{P_{1},P_{2},\ldots,P_{10}\}$. Let
us check that $\PP$ satisfies all hypothesis of
Lemma~\ref{lemma:10-points-on-P2}. It follows from
Lemma~\ref{lemma:Bordiga-anticanclass-empty} that there are no
cubic curves passing through all points of $\PP$.

\begin{lemma}
\label{lemma:Borgiga-final} Let $C$ be a $\bar{G}$-invariant curve
in $\P^2$ that passes through at least $r$ of the points of $\PP$.
Then $4\deg(C)-r\ge 4$.
\end{lemma}

\begin{proof}
Suppose that $4\deg(C)-r\leqslant 3$. Let us show that this
assumption leads to a contradiction.

Let $\tilde{C}$ be the proper transform of the curve $C$ on the
surface $\tilde{S}$. Then $\tilde{C}$ is $\bar{G}$-invariant,
which implies that the set $f(\tilde{C})$ is also
$\bar{G}$-invariant.

Since $\deg(C)\leqslant 3/4+r/4\leqslant 3/4+10/4=13/4$, we see
that $\deg{C}\leqslant 3$. Hence $C$ consists of at most three
irreducible components. In particular, it follows from
$(\mathrm{G})$ that no components of the curve $\tilde{C}$ are
contracted by  $f$. Hence, we see that $f(\tilde{C})$ is a curve
in $S\subset\P^4$ of degree
\begin{multline}
\label{equation:Borgiga-final}
f^{*}\big(H\big)\cdot\tilde{C}=
\Big(\pi^{*}\Big(\mathcal{O}_{\P^2}\big(4\big)\Big)-
\sum_{i=1}^{10}E_{i}\Big)\cdot\tilde{C}=\\
=4\deg\big(C\big)-\sum_{i=1}^{10}\mathrm{mult}_{P_{i}}\big(C\big)\leqslant
4\deg\big(C\big)-r\leqslant 3,%
\end{multline}
which immediately implies that $\tilde{C}$ is reducible by
$(\mathrm{G})$, because irreducible curve of degree at most $3$ is
contained in a hyperplane in $\P^4$. Moreover, we see that
$f(\tilde{C})$ is a union of three different lines by
$(\mathrm{G})$, because any curve of degree at most $2$ is
contained in a hyperplane in $\P^4$. Since $\deg{C}\leqslant 3$,
we see that $C$ is a union of three lines. Then it follows from
$(\ref{equation:Borgiga-final})$ that
$$
3=f^{*}\big(H\big)\cdot\tilde{C}=12-\sum_{i=1}^{10}\mathrm{mult}_{P_{i}}\big(C\big)\leqslant 12-r\leqslant 3,%
$$
which implies that $r=9$. Thus, the curve $C$ passes through
exactly $9$ points in $\PP$, which implies that (at least) one
curve among $f(E_{i})$, $1\le i\le 10$, is $\bar{G}$-invariant,
which is impossible by $(\mathrm{G})$, because the curves
$f(E_{i})$ are lines in $S\subset\P^4$.
\end{proof}

It follows from $(\mathrm{G})$ that there are no $\bar{G}$-orbits
of length at most $4$ contained in $S$. Thus, there are no
$\bar{G}$-orbits of length at most $4$ contained in $\tilde{S}$,
which implies that there are no $\bar{G}$-orbits of length at most
$4$ contained in $\P^2\setminus\PP$. Similarly, it follows from
$(\mathrm{G})$ that there are no $\bar{G}$-orbits of length at
most $2$ on $\PP$, because the curves $f(E_{i})$ are lines in
$S\subset\P^4$. Therefore, we see that~$\PP$ satisfies all
hypothesis of Lemma~\ref{lemma:10-points-on-P2}. Thus, the divisor
$-K_{\tilde{S}}$ is not big by Lemma~\ref{lemma:10-points-on-P2},
which completes the proof of Theorem~\ref{theorem:no-Bordiga}.

\section{Six-dimensional singularities}
\label{section:6-dim}

The main purpose of this section is to prove
Theorem~\ref{theorem:main-6}. We start with an easy observation.

\begin{lemma}
\label{lemma:S4-large-representations} Let $G$ be a transitive
finite subgroup in $\GL_{n+1}(\mathbb{C})$ such that
$\bar{G}\cong\SS_4$. Then $n\leqslant 3$.
\end{lemma}

\begin{proof}
Since the Schur multiplier of the group \mbox{$\bar{G}\cong\SS_4$}
is $\Z_2$, we may assume  that either \mbox{$G\cong\SS_4$} or
\mbox{$G\cong 2.\SS_4$}. Therefore, there is an irreducible
$(n+1)$-dimensional linear representation \mbox{$V\cong\C^{n+1}$}
of either the group $2.\SS_4$, or the group $\SS_4$. Since the
group $\SS_4$ does not have irreducible representations of
dimension greater than $4$, we may assume that $G\cong 2.\SS_4$
and the center of the group $2.\SS_4$ acts non-trivially on $V$.
Then $(n+1)^2\leqslant |2.\SS_4|-|\SS_4|=24$, which implies that
$n\leqslant 3$.
\end{proof}

Now we are ready to prove Theorem~\ref{theorem:main-6}. Let $G$ be
a~finite subgroup in $\GL_{6}(\mathbb{C})$ that does not contain
reflections, and let
$\phi\colon\GL_{6}(\mathbb{C})\to\PGL_{6}(\mathbb{C})$ be
the~natural projection. Put $\bar{G}=\phi(G)$. Let us identify
$\PGL_{6}(\mathbb{C})$ with $\mathrm{Aut}(\mathbb{P}^{5})$.
Suppose that $G$ is transitive, the~group $G$ does not have
semi-invariants of degree at most $5$, there is no irreducible
$\bar{G}$-invariant smooth rational cubic scroll in
$\mathbb{P}^{5}$, and there is no irreducible  $\bar{G}$-invariant
complete intersection of two quadric hypersurfaces
in~$\mathbb{P}^{5}$.

\begin{theorem}
\label{theorem:6-dim-1}  If $\C^{6}\slash G$ is not
weakly-exceptional, then there exists an irreducible
$\bar{G}$-invariant, normal, projectively normal, non-degenerate
Fano type threefold $X\subset\mathbb{P}^{5}$ of degree $6$ and
sectional genus~$3$ such that
$h^{0}(\mathcal{O}_{\mathbb{P}^{5}}(2)\otimes\mathcal{I}_{X})=0$
and
$h^{0}(\mathcal{O}_{\mathbb{P}^{5}}(3)\otimes\mathcal{I}_{X})=4$.
\end{theorem}

In the remaining part of this section we will prove
Theorem~\ref{theorem:6-dim-1}, which implies
Theorem~\ref{theorem:main-6}. Suppose that $\C^{6}\slash G$ is not
weakly-exceptional. By
Theorem~\ref{theorem:weakly-exceptional-quotient-criterion},
there is an irreducible $\bar{G}$-invariant, irreducible, normal,
projectively normal Fano type subvariety $X\subset\P^5$~such that
$$
\mathrm{deg}\big(X\big)\leqslant {5\choose \mathrm{dim}\big(X\big)},%
$$
and $X$ has more additional properties that we are about to
describe. Let $\mathcal{I}_{X}$ be the ideal sheaf of the
subvariety $X\subset\P^5$. Then
$h^{i}(\mathcal{O}_{\P^5}(m)\otimes \mathcal{I}_{X})=0$ for every
$i\geqslant 1$ and for every $m\geqslant 0$. Let $H$ be a general
hyperplane section of the subvariety $X\subset\P^5$. If
$\mathrm{dim}(X)\geqslant 2$, then $H$ is irreducible,
projectively normal and
$h^{i}(\mathcal{O}_{H}\otimes\mathcal{O}_{\mathbb{P}^{5}}(m))=0$
for every $i\geqslant 1$ and $m\geqslant 1$. Since~$G$ is
transitive and $X$ is $\bar{G}$-invariant, the subvariety $X$ is
not contained in a~hyperplane in $\mathbb{P}^{5}$. In particular,
we see that $X$ is not a point. Since $G$ does not have
semi-invariants of degree at most $5$, we see that $X$ is not a
fourfold.

\begin{lemma}
\label{lemma:invariant-lift} The linear system $|mH|$ does not
have $\bar{G}$-invariant divisors for every positive
integer~$m\leqslant 5$.
\end{lemma}

\begin{proof}
Since $X$ is projectively normal, the sequence
\begin{equation}
\label{eqaution:invariant-lift}
0\longrightarrow H^{0}\Big(\mathcal{O}_{\P^5}\big(m\big)\otimes\mathcal{I}_{X}\Big)\longrightarrow
H^{0}\Big(\mathcal{O}_{\P^5}\big(m\big)\Big)\longrightarrow H^{0}\Big(\mathcal{O}_{X}\big(mH\big)\Big)\longrightarrow 0%
\end{equation}
is exact. We can consider $(\ref{eqaution:invariant-lift})$ as
an~exact sequence~of~linear $G$-representations. If $|mH|$
contains a $\bar{G}$-invariant divisor, then $G$ has a
semi-invariant of degree $m$. But $G$ does not have
semi-invariants of degree at most $5$ by assumption.
\end{proof}

Unfortunately, we can not claim now that $X$ has Kawamata log
terminal singularities, since~$-K_{X}$ is not necessary a Cartier
divisor. Nevertheless, we know that $X$ has at most rational
singularities (see \cite[Theorem~1.3.6]{KMM}). If $X$ is a
surface, then $X$ has quotient singularities.

\begin{lemma}
\label{lemma:dim-6-curves} The subvariety $X$ is not a curve.
\end{lemma}

\begin{proof}
If $X$ is a curve, then $X$ is a smooth rational curve of degree
$5$, because $\deg(X)\le 5$ and~$X$ is not contained in
a~hyperplane in $\mathbb{P}^{5}$. Then $\bar{G}$ acts faithfully
on $X$. If $\bar{G}\not\cong\SS_{4}$, then there is a
$\bar{G}$-invariant effective divisor in $\mathrm{Pic}(\P^1)$ of
degree $20$, which is impossible by
Lemma~\ref{lemma:invariant-lift}. Since the~group $G$ does not
have semi-invariants of degree at most $5$, we see that
$\bar{G}\cong\SS_{4}$, which contradicts
Lemma~\ref{lemma:S4-large-representations}, because $G$ is
transitive.
\end{proof}

\begin{lemma}\label{lemma:pencil-of-quadrics}
If
$h^{0}(\mathcal{O}_{\mathbb{P}^{5}}(2)\otimes\mathcal{I}_{X})\ne
0$, then
$h^{0}(\mathcal{O}_{\mathbb{P}^{5}}(2)\otimes\mathcal{I}_{X})\geqslant
3$.
\end{lemma}

\begin{proof}
Since the~group $G$ does not have semi-invariants of degree at
most $5$, we see that
$h^{0}(\mathcal{O}_{\mathbb{P}^{5}}(2)\otimes\mathcal{I}_{X})\ne
1$. Thus, we must prove that
$h^{0}(\mathcal{O}_{\mathbb{P}^{5}}(2)\otimes\mathcal{I}_{X})\ne
2$.

Suppose that
$h^{0}(\mathcal{O}_{\mathbb{P}^{5}}(2)\otimes\mathcal{I}_{X})=2$.
Then there exists a $\bar{G}$-invariant pencil of quadrics in
$\P^5$ whose base locus contains $X$. Let us denote this pencil by
$\mathcal{P}$. Since the~group $G$ does not have semi-invariants
of degree at most $5$ and $G$ is transitive, the base locus of the
pencil $\mathcal{P}$ is an irreducible threefold that is a
$\bar{G}$-invariant complete intersection of two quadrics in
$\mathcal{P}$, which is impossible by assumption.
\end{proof}

Since $X$ is projectively normal, there is an~exact sequence of
cohomology groups
\begin{equation}
\label{equation:exact-sequence-n-6}
0\longrightarrow H^{0}\Big(\mathcal{O}_{\mathbb{P}^{5}}\big(n\big)\otimes\mathcal{I}_{X}\Big)\longrightarrow H^{0}\Big(\mathcal{O}_{\mathbb{P}^{5}}\big(n\big)\Big)\longrightarrow H^{0}\Big(\mathcal{O}_{X}\big(nH\big)\Big)\longrightarrow 0%
\end{equation}
for every $n\geqslant 1$. Let $f\colon\tilde{X}\to X$ be the
resolution of singularities of the variety $X$. If $X$ is a
surface, then we assume that $f$ is a minimal resolution of
singularities. Put $\tilde{H}=f^*(H)$. Then~$\tilde{H}$ is nef and
big. Hence, it follows from the Nadel--Shokurov vanishing theorem
that
\begin{equation}
\label{equation:RRR}
\chi\Big(\mathcal{O}_{\tilde{X}}\big(K_{\tilde{X}}+\tilde{H}\big)\Big)=h^0\Big(\mathcal{O}_{\tilde{X}}\big(K_{\tilde{X}}+\tilde{H}\big)\Big).
\end{equation}

\begin{lemma}\label{lemma:d-4-q-6}
If $X$ is a surface, then $\deg(X)\ne 4$.
\end{lemma}

\begin{proof}
Suppose that $X$ is a surface and $\deg(X)=4$. Let us use the
usual notation for the rational normal scrolls
(see e.\,g.~\cite[Example~8.2.6]{Harris}). Since $X$ is non-degenerate, it
follows from \cite{EisenbudHarris} that $X$ is either a cone over
a rational normal curve in $\P^4$, or a rational normal scroll
$X_{1,3}$, or a rational normal scroll $X_{2,2}$, or a Veronese
surface. The first two cases are impossible since otherwise there
would exist a $\bar{G}$-invariant point or a $\bar{G}$-invariant
line in $\P^5$, respectively, which would contradict the
transitivity of $G$. If $X$ is a rational normal scroll $X_{2,2}$,
then the family of linear spans of the one-parameter family of
conics on $X$ sweeps out an irreducible $\bar{G}$-invariant smooth
rational cubic scroll, which is impossible, since we assume that
there is no such threefolds in $\mathbb{P}^{5}$. Thus, we see that
$X$ is a Veronese surface. Then it follows from
\cite[Exercise~8.8]{Harris} that the secant variety of the
surface $X$ is a cubic hypersurface in $\P^5$, which must be
$\bar{G}$-invariant, because $X$ is $\bar{G}$-invariant. The
latter is impossible, since we assume that the~group $G$ does not
have semi-invariants of degree at most $5$.
\end{proof}

\begin{lemma}\label{lemma:d-5-q-5}
If $X$ is a surface and $\deg(X)=5$, then $-K_X\cdot H\ne 5$.
\end{lemma}
\begin{proof}
Suppose that $X$ is a surface such that $\deg(X)=5$ and $-K_X\cdot
H=5$. It follows from $(\ref{equation:RRR})$ and the Riemann--Roch
theorem that
$$
h^0\Big(\mathcal{O}_{\tilde{X}}\big(K_{\tilde{X}}+\tilde{H}\big)\Big)=1+\frac{\big(K_{\tilde{X}}+\tilde{H}\big)\cdot\tilde{H}}{2}=1,
$$
which means that there is a unique effective divisor $D$ in the
linear system $|K_{\tilde{X}}+\tilde{H}|$. Moreover, since
$D\cdot\tilde{H}=0$, the support of the divisor $D$ is contained
in the exceptional locus of the morphism $f$. On the other hand,
for any effective divisor $E$ supported in the exceptional locus
of the morphism $f$, one has $K_{\tilde{X}}\cdot E\ge 0$ and
$E\cdot\tilde{H}=0$, because we assume that $f$ is a minimal
resolution of singularities. Therefore, we must have  $D\sim 0$.
Since, $K_{X}+H\sim f_*(K_{\tilde{X}}+\tilde{H})$, we see that
$K_{\tilde{X}}\sim -\tilde{H} \sim f^*(K_X)$, which implies that
$X$ is a del Pezzo surface with $K_{X}^{2}=5$ that has at most Du
Val singularities. Then $X$ cannot have more than $4$ singular
points, which implies that $X$ is smooth, because the group $G$ is
transitive. We see that $X$ is an anticanonically embedded smooth
del Pezzo surface of degree $5$. Thus the action of $\bar{G}$
lifts isomorphically to the vector space
$H^0(\O_X(H))=H^0(\O_X(-K_{X}))\cong\C^6$. Moreover, it follows
from the transitivity of the group $G$ that $\bar{G}$ acts
faithfully on the surface $X$. Thus, the group $\bar{G}$ is a
subgroup of the symmetric group $\SS_5$ acting on $\C^6$, since
$\mathrm{Aut}(X)\cong\SS_{5}$. The sum of all $(-1)$-curves in $X$
is a $\bar{G}$-invariant divisor that is linearly equivalent to
$-2K_{X}$ (cf. \cite[Lemma~5.7]{Ch07b}), which is impossible by
Lemma~\ref{lemma:invariant-lift}.
\end{proof}

\begin{lemma}
\label{lemma:d-6-q-4} If $X$ is a surface and $\deg(X)=6$, then
$-K_X\cdot H\ne 4$.
\end{lemma}
\begin{proof}
Suppose that $X$ is a surface such that $\deg(X)=6$ and $-K_X\cdot
H=4$. It follows from~$(\ref{equation:RRR})$ and the Riemann--Roch
theorem that
$$
h^0\Big(\mathcal{O}_{\tilde{X}}\big(K_{\tilde{X}}+\tilde{H}\big)\Big)=1+\frac{\big(K_{\tilde{X}}+\tilde{H}\big)\cdot\tilde{H}}{2}=2.
$$

Let $\tilde{C_1}$ and $\tilde{C_2}$ be general curves in
$|K_{\tilde{X}}+\tilde{H}|$. Put $C_1=f(\tilde{C_1})$ and
$C_2=f(\tilde{C_2})$. Then every curve in the pencil $|K_X+H|$ is
a curve of degree at most $2$, since $H\cdot C_1=H\cdot C_{2}=2$.
In particular, either the base locus of the linear system
$|K_X+H|$ contains a $\bar{G}$-invariant curve of degree at most
$2$, or the intersection $C_1\cap C_2$ consists of at most $4$
points. Since $G$ is transitive, we must have $C_1\cap
C_2=\varnothing$, so that $|K_X+H|$ is base point free. In
particular, the divisor~$K_X$ is Cartier, which implies that $X$
has at most Du Val singularities. Furthermore, one has
$(K_X+H)^2=0$, which implies that $K_X^2=2$, and the pencil
$|K_X+H|$ induces a morphism $\vartheta\colon X\to\P^1$ whose
general fiber is a smooth rational curve. Then the morphism
$\vartheta$ induces a homomorphism
$\theta\colon\bar{G}\to\mathrm{Aut}(\P^1)$.

Note that $h^2(\O_X(-K_X))=h^0(\O_X(2K_X))$ by the Serre duality,
and $h^0(\O_X(2K_X))=0$ since $H\cdot K_X<0$. Then
$h^0(\O_X(-K_X))\geqslant 1+K_X^2=3$ by the Riemann--Roch theorem.
Let us show that the base locus of the linear system $|-K_{X}|$
does not contain curves.

Suppose that the base locus of the linear system $|-K_{X}|$
contains curves. Then the base locus of the linear system $|-K_X|$
contains a $\bar{G}$-invariant curve. Let us denote it by $E$.
Then $\deg(E)\leqslant 3$, because $\deg(E)\leqslant -K_X\cdot
H=4$ and $h^0(\O_X(-K_X))>1$. Since $G$ is transitive, the only
possibility is that $E$ is a disjoint union of three lines in
$\P^5$. Let us denote these lines by $L_1$, $L_2$, and $L_3$. Then
$G$ acts transitively on the set $\{L_1, L_2, L_3\}$. Moreover,
there is a line $L$ in $X\subset\P^5$ such that $-K_X\sim
L_1+L_2+L_3+L$, and $|L|$ is $\bar{G}$-invariant and free from
base curves. Applying the adjunction formula to $L$, we see that
$-2=(K_X+L)\cdot L=-3L_1\cdot L$, which is a contradiction. Hence,
the base locus of the linear system $|-K_X|$ does not contain
curves.

Since the base locus of the linear system $|-K_X|$ does not
contain curves, the divisor $-K_X$ is nef. Furthermore, the
divisor $-K_X$ is big since $(-K_X)^2=2$, which implies that $X$
is a weak del Pezzo surface of degree $2$. Let $\sigma\colon
X\to\bar{X}$ be the blow-down of all curves that has trivial
intersection with $-K_{X}$. Then $\bar{X}$ is a possibly singular
del Pezzo surface of degree $2$, i.e. $-K_{\bar{X}}$ is ample and
$K_{\bar{X}}^{2}=2$. Moreover, the action of the group $\bar{G}$
on the surface $X$ induces a faithful action of the group
$\bar{G}$ on the surface $\bar{X}$.

Let $\kappa\colon\bar{X}\to\P^2$ be the double cover that is given
by the linear system $|-K_{\bar{X}}|$, let $R$ be the ramification
divisor in $\P^2$ of the double cover $\kappa$, let $\bar{R}$ be
the curve in $\bar{X}$ such that $\kappa(\bar{R})=R$, and let
$\xi$ be the involution in $\mathrm{Aut}(\bar{X})$ that is induced
by $\kappa$. Then $\xi$ induces an~exact sequence of groups
$$
\xymatrix{
1\ar@{->}[rr]&&\mathbb{Z}_{2}\ar@{->}[rr]^{\alpha}&& \mathrm{Aut}\big(\bar{X}\big)\ar@{->}[rr]^{\beta}&&\Gamma\ar@{->}[rr]&&1,}%
$$
where $\mathrm{im}(\alpha)=\langle\xi\rangle$, $\Gamma$ is a
subgroup in $\mathrm{Aut}(\P^2)$ such that $R$ is
$\Gamma$-invariant, and $\beta$ is induced by the double cover
$\kappa$. Since $\bar{X}$ is a quartic hypersurface in
$\P(1,1,1,2)$ and $\kappa$ is induced by the natural projection
$\P(1,1,1,2)\to\P^2$, this sequence splits, so we have
$\mathrm{Aut}(\bar{X})\cong \Gamma\times\mathbb{Z}_{2}$. Note that
$\xi$ is known as the Geiser involution of the surface $\bar{X}$.

Suppose that $X$ is singular. Then it follows from
Theorem~\ref{theorem:Noether} that
\begin{equation}\label{eq:Noether-6-or-7}
\rk\Pic(X)+K_X^2+\sum_{P\in X}\mu\big(P\big)=10,
\end{equation}
where $\mu(P)$ is a Milnor number of a point $P\in X$. Since $G$
is transitive, we see that \mbox{$|\mathrm{Sing}(X)|\geqslant 6$},
and all singular points of the surface $X$ are ordinary double
points. By~$(\ref{eq:Noether-6-or-7})$ one has
\mbox{$|\mathrm{Sing}(X)|\leqslant 7$}. If $|\mathrm{Sing}(X)|=7$,
then $\rk\Pic(X)=1$, which implies that $X\cong\bar{X}$ is a del
Pezzo surface of degree $2$ that has $7$ singular points, which
easily leads to a contradiction. But $G$ acts transitively on the
set of singular points of the surface $X$ (indeed, otherwise
$\bar{G}$ would have an orbit of length at most $3$, which is
impossible since $G$ is transitive). Thus, we see that
$|\mathrm{Sing}(X)|=6$, which implies that either $-K_X$ is
already ample, or there exists an irreducible curve $Z\subset X$
such that $-K_X\cdot Z=0$. In the latter case
$\mathrm{Sing}(X)\cap C\ne\varnothing$, since $\bar{X}$ can not
have more than $6$ singular points. Thus, if there exists an
irreducible curve $Z\subset X$ such that $-K_X\cdot Z=0$, then
$Z\cong\P^1$ and $Z$ contains all singular points of the surface
$X$. Applying the adjunction formula to the proper transform of
the curve $Z$ on the surface $\tilde{X}$, we see that $Z^{2}=1$,
which is impossible, since $Z^{2}<0$, because $Z$ is contracted
by~$\sigma$. Thus, we see that $-K_X$ is ample, so that
$X=\bar{X}$ is a del Pezzo surface of degree $2$ with $6$ singular
points. In fact, such del Pezzo surface is unique, since $R$ must
be a union of $4$ lines in general position. Since $\rk\Pic(X)=2$
by $(\ref{eq:Noether-6-or-7})$, every fiber of the morphism
$\vartheta$ is an irreducible rational curve. Since~$X$ has only
ordinary double points, every fiber of $\vartheta$ contains either
none or two singular points of the surface $X$. Thus, six singular
points of the surface $X$ must split in pairs, which implies that
there is a $\theta(\bar{G})$-orbit in $\P^1$ that contains at most
$3$ points. Then there is a $\theta(\bar{G})$-orbit in $\P^1$ that
contains at most $2$ points. In particular, there exist two fibres
$F_{1}$ and $F_{2}$ of the morphism $\vartheta$ such that the
divisor $F_{1}+F_{2}$ is $\bar{G}$-invariant. Hence
$F_{1}+F_{2}+\bar{R}$ is a $\bar{G}$-invariant divisor as well,
but $F_{1}+F_{2}+\bar{R}\sim 2H$, which is impossible by
Lemma~\ref{lemma:invariant-lift}. The obtained contradiction shows
that the surface $X$ can not be singular.

Therefore, the surface $X$ is smooth. Then every fiber of the
morphism $\vartheta$ is reduced, and $\vartheta$ has exactly $6$
reducible fibers, which consist of two smooth rational curves
intersecting transversally in one point. In particular, we see
that there exists a $\bar{G}$-orbit in $X$ that consists of $6$
points, since there are no $\bar{G}$-orbits in $X$ that consist of
at most $5$ points, because $G$ is transitive. Then there exists a
$\theta(\bar{G})$-orbit in $\P^1$ that consists of $6$ points,
which implies that $\theta(\bar{G})\not\cong\A_5$. In particular,
we see that $\bar{G}\not\cong\A_5$. We must consider two cases:
$X\cong\bar{X}$ and $X\not\cong\bar{X}$.

Suppose that $X\cong\bar{X}$ and  $\theta(\bar{G})\not\cong\SS_4$.
Then there exists a $\theta(\bar{G})$-orbit in $\P^1$ that
contains either $1$ or $2$ or $4$ points, because we already
proved that $\theta(\bar{G})\not\cong\A_5$. Thus, there are $4$
(not necessary distinct) fibers $\bar{F}_{1}$, $\bar{F}_{2}$,
$\bar{F}_{3}$, $\bar{F}_{4}$ of the morphism $\vartheta$ such that
the divisor $\sum_{i=1}^{4}\bar{F}_{i}$ is $\bar{G}$-invariant.
Then $2\bar{R}+\sum_{i=1}^{4}\bar{F}_{i}\sim 4H$, and the divisor
$2\bar{R}+\sum_{i=1}^{4}\bar{F}_{i}$ is $\bar{G}$-invariant, which
is impossible by Lemma~\ref{lemma:invariant-lift}. The obtained
contradiction shows that $\theta(\bar{G})\cong\SS_4$ if
$X\cong\bar{X}$.

Suppose that $X\cong\bar{X}$. Then $\theta(\bar{G})\cong\SS_4$.
Moreover, it follows from
Lemma~\ref{lemma:S4-large-representations} that $\bar{G}\not\cong
\SS_4$, which implies that $|\bar{G}|\geqslant 48$. Since
$X\slash\langle\xi\rangle\cong\mathbb{P}^{2}$, we see that
$\langle\xi\rangle$-invariant subgroup of the group $\Pic(X)$ is
$\mathbb{Z}$, which implies that $\xi\not\in\bar{G}$, since the
pencil $|K_{X}+H|$ is $\bar{G}$-invariant. Thus, we see that
$\bar{G}\cong\beta(\bar{G})$. Browsing through the list of the
possible automorphism groups of smooth del Pezzo surfaces of
degree $2$ (see~\cite[Table~6]{DoIs06}) and using the fact that
$|\bar{G}|$ is divisible by $24$ and $|\bar{G}|\geqslant 48$, we
see that $\Gamma$ is a subgroup of the following groups:
$\PSL_2(\F_7)$, $(\Z_4\times\Z_4)\rtimes\SS_3$, or $4.\A_4$. On
the other hand, the only subgroup of the group $\PSL_2(\F_7)$ that
admits a surjective homomorphism to $\SS_4$ is isomorphic to
$\SS_4$ (see \cite{Atlas}), which implies that either
$\Gamma\cong(\Z_4\times\Z_4)\rtimes\SS_3$ or $\Gamma\cong 4.\A_4$.
If $\Gamma\cong 4.\A_4$, then the center of the group $\bar{G}$
must be contained in the kernel of the homomorphism~$\psi$,
because the group $\SS_4$ has trivial center. Thus, if
$\Gamma\cong 4.\A_4$, then the monomorphism $\beta$ composed with
the natural epimorphism $4.\A_4\to\A_4$ gives a monomorphism
$\bar{G}\to\A_4$, which is impossible, since $|\bar{G}|\geqslant
48$. Thus, we see that $\Gamma\cong (\Z_4\times\Z_4)\rtimes\SS_3$.
Then it follows from the existence of the epimorphism
$\psi\colon\bar{G}\to\SS_4$, that either
$\bar{G}\cong\beta(\bar{G})=\Gamma\cong(\Z_4\times\Z_4)\rtimes\SS_3$,
or $\Gamma\cong\SS_4$ (cf. the proof of
\cite[Theorem~6.17]{DoIs06}). Since we already know that
$\Gamma\not\cong\SS_4$, then $\bar{G}\cong\Gamma$. It follows from
the proof of \cite[Theorem~6.17]{DoIs06} that the
$\bar{G}$-invariant subgroup of the Picard group is $\mathbb{Z}$,
which is a contradiction, since  $|K_{X}+H|$ is
$\bar{G}$-invariant. The obtained contradiction shows that
$X\not\cong\bar{X}$.

Since $X\not\cong\bar{X}$, we see that the divisor $-K_X$ is not
ample. Let $E_{1},\ldots,E_{r}$ be all irreducible curves that are
contracted by $\sigma$, let $C_{1}$ and $C_{2}$ be general curves
in the pencil $|K_X+H|$. Then~$C_{1}$ and~$C_{2}$ are smooth
irreducible curves, and $R_{i}\cdot E_{j}\geqslant 1$ for all $i$
and $j$, since none of the curves $E_{1},\ldots,E_{r}$ can not be
a component of any curve in the pencil $|K_X+H|$. Since
$$
-K_{\bar{X}}\cdot\sigma(C_{1})=-K_{X}\cdot C_{1}=-K_{\bar{X}}\cdot\sigma(C_{2})=-K_{X}\cdot C_{2}=2,%
$$
we see that either the curves $\sigma(C_{1})$ and $\sigma(C_{2})$
are smooth rational curves such that  $\kappa\circ\sigma(C_{1})$
and $\kappa\circ\sigma(C_{2})$ are conics in $\P^2$, or
$\kappa\circ\sigma(C_{1})$ and $\kappa\circ\sigma(C_{2})$ are
lines in $\P^2$.

Suppose that $\kappa\circ\sigma(C_{1})$ and
$\kappa\circ\sigma(C_{2})$ are lines in $\P^2$. Then the
intersection
$\kappa\circ\sigma(C_{1})\cap\kappa\circ\sigma(C_{2})$ consists of
a single point in $R$, which implies that
$\sigma(C_{1})\cap\sigma(C_{2})$ also consists of a single point
that must be the unique singular point of the surface $\bar{X}$.
So, we see that $\sigma(E_{1})=\ldots=\sigma(E_{r})$. If this
point is not an ordinary double point of the surface $\bar{X}$,
then $X\subset\P^5$ contains a $\bar{G}$-orbit of length at most
$3$,  which is impossible, because $G$ is transitive. Hence, we
see that $r=1$ and $\bar{X}$ has one singular point, which is an
ordinary double point. In particular, the curve $E_{1}$ is
$\bar{G}$-invariant. On the other hand, the curves $\sigma(C_{1})$
and $\sigma(C_{2})$ are curves in the linear
system~\mbox{$|-K_{\bar{X}}|$} that are singular at the point $\sigma(E_{1})$.
Moreover, the multiplicity of the curves $\sigma(C_{1})$ and
$\sigma(C_{2})$ at the point $\sigma(E_{1})$ must equal $2$, which
implies that $E_{1}\cdot C_{1}=E_{1}\cdot C_{2}=2$.  Then
$$
H\cdot E_{1}=C_{1}\cdot E_{1}-K_X\cdot E_{1}=C_{2}\cdot E_{1}-K_X\cdot E_{1}=2,%
$$
which is impossible, since $G$ is transitive. The obtained
contradiction shows that the curves $\kappa\circ\sigma(C_{1})$ and
$\kappa\circ\sigma(C_{2})$ are not lines in $\P^2$.

We see that $\kappa\circ\sigma(C_{1})$ and
$\kappa\circ\sigma(C_{2})$ are conics in $\P^2$. Then the curves
$\sigma(C_{1})$ and $\sigma(C_{2})$ are smooth, which implies that
$C_{i}\cdot E_{j}=1$ for every $i$ and $j$. Then $H\cdot
E_{i}=C_{j}\cdot E_{i}-K_X\cdot E_{i}=1$ for every $i$ and $j$,
which implies that $r\geqslant 3$, since the group $G$ is
transitive. If $\bar{X}$ has a $\bar{G}$-invariant singular point
that is not an ordinary double point of the surface $\bar{X}$,
then $X$ contains a $\bar{G}$-orbit of length at most $3$, which
is impossible, because $G$ is transitive. On the other hand, we
know that $\rk\Pic(\bar{X})+r=8$ by Theorem~\ref{theorem:Noether},
which implies that $3\leqslant r\leqslant 7$. Since the singular
points of the surface $\bar{X}$ are Du Val singular points, we see
that either $\bar{X}$ has only ordinary double points, or the set
$\Sing(\bar{X})$ consists of $2$ singular points of
type~$\mathbb{A}_3$, or the set $\Sing(\bar{X})$ consists of $3$
singular points of type~$\mathbb{A}_2$. However, if the set
$\Sing(\bar{X})$ consists of $2$ singular points of
type~$\mathbb{A}_3$, then $X$ contains a $\bar{G}$-orbit of length at
most $4$,  which is impossible, because $G$ is transitive.
Similarly, if the set $\Sing(\bar{X})$ consists of $3$ singular
points of type $\mathbb{A}_2$, then $X$ contains a $\bar{G}$-orbit
of length at most $3$, which is impossible, because $G$ is
transitive. Therefore, we see that $\bar{X}$ has only ordinary
double points, which implies that the curves $E_{1},\ldots,E_{r}$
are disjoint and all points
$\kappa\circ\sigma(E_{1}),\ldots,\kappa\circ\sigma(E_{r})$ are
different. Note that
$\kappa\circ\sigma(E_{1}),\ldots,\kappa\circ\sigma(E_{r})$ are all
singular points of the curve $R$, and they are ordinary double
points of the curve $R$. Since
$$
\Big\{\kappa\circ\sigma(E_{1}),\ldots,\kappa\circ\sigma(E_{r})\Big\}\subset\kappa\circ\sigma(C_{1})\cap \kappa\circ\sigma(C_{2}),%
$$
we see that $r\leqslant 4$, because  the intersection
$\kappa\circ\sigma(C_{1})\cap \kappa\circ\sigma(C_{2})$ consists
of at most four points, and the curves $\kappa\circ\sigma(C_{1})$
and $\kappa\circ\sigma(C_{2})$ are conics in $\P^2$.

Suppose that there is a point $P\in\P^2$ that is fixed by
$\beta(\bar{G})$. If $P\not\in\Sing(R)$, then $X$ has
$\bar{G}$-orbit of length at most $2$, which is impossible,
because $G$ is transitive. If $P=\kappa\circ\sigma(E_{i})$ for
some $i\in\{1,\ldots,r\}$, then the intersection of the curve
$E_{i}$ and a proper transform of the curve~$\bar{R}$ on the
surface $X$ is a $\bar{G}$-invariant set of two points, which is
impossible, since~$G$ is transitive. We see that there is no
$\beta(\bar{G})$-invariant point in $\P^2$.

Suppose that $r=3$. Then $\bar{G}$ acts transitively on
$|\Sing(R)|$, since there are no $\beta(\bar{G})$-fixed points in
$\P^2$. One has $\sigma(C_{1})\cdot
\sigma(C_{1})=\sigma(C_{2})\cdot \sigma(C_{2})= 3/2$, which
implies that the pencil $|K_{X}+H|$ is not $\xi$-invariant. In
particular, we see that $\xi\not\in\bar{G}$, which implies that
$\bar{G}\cong\beta(\bar{G})$. Note that either $R$ is a union of a
line and a nonsingular cubic curve, or $R$ is an irreducible plane
quartic curve with three ordinary double points. In the former
case each component of the curve $R$ must be
$\beta(\bar{G})$-invariant, which implies that there exists a
point in $\P^2$ that is fixed by $\beta(\bar{G})$, which is
impossible. Thus, we see that $R$ is an irreducible plane quartic
curve with three ordinary double points. In particular, the group
$\beta(\bar{G})$ acts faithfully on the curve $R$. Let
$\upsilon\colon\tilde{R}\to R$ be the normalization of the curve
$R$. Then the faithful action of the group $\beta(\bar{G})$ on the
curve $R$ induces a faithful action of the group $\beta(\bar{G})$
on the curve $\tilde{R}$. On the other hand, the curve $\tilde{R}$
is rational, which implies that either $\beta(\bar{G})\cong\SS_4$,
or $\tilde{R}$ contains a $\beta(\bar{G})$-orbit of length at most
$2$. The former case is impossible by
Lemma~\ref{lemma:S4-large-representations}, since
$\bar{G}\cong\beta(\bar{G})$. Hence, there is a
$\beta(\bar{G})$-orbit $\Omega\subset\tilde{R}$ that consists of
at most $2$ points. Then $\upsilon(\Omega)\not\subset\Sing(R)$,
since otherwise one of the points of $\Sing(R)$ would be fixed by
$\beta(\bar{G})$, which is impossible. Hence, the surface $X$
contains a $\bar{G}$-orbit that consists of at most $2$ points,
which is impossible, since $G$ is transitive.

Thus, we see that $r=4$. Then either $R$ is a union of a line and
an irreducible singular cubic curve, or $R$ is a union of two
irreducible conics that intersect each other transversally. In the
former case the exists a $\beta(\bar{G})$-fixed point in $\P^2$,
which is impossible. Thus, we see that $R$ is union of two
irreducible conics. Let us call them $T_1$ and $T_2$. Let
$\tilde{T}_{1}$ and and $\tilde{T}_{2}$ be irreducible curves in
$X$ such that $\kappa\circ\sigma(\tilde{T}_{1})=T_{1}$ and
$\kappa\circ\sigma(\tilde{T}_{2})=T_{2}$. Then $\tilde{T}_{1}\cdot
C_{1}=\tilde{T}_{1}\cdot C_{2}=\tilde{T}_{2}\cdot
C_{1}=\tilde{T}_{2}\cdot C_{2}=0$, which implies that the curves
$\tilde{T}_{1}$ and $\tilde{T}_{2}$ are contained in the pencil
$|K_{X}+H|$. Note that in this case the pencil $|K_{X}+H|$ is
$\langle\xi\rangle$-invariant, so we do not know whether $\xi$ is
contained in $\bar{G}$ or not. On the other hand, we have $2H\sim
2\tilde{T}_{1}+2\tilde{T}_{2}+\sum_{i=1}^{4}E_{i}$, and the
divisor $2\tilde{T}_{1}+2\tilde{T}_{2}+\sum_{i=1}^{4}E_{i}$ is
$\bar{G}$-invariant, which is impossible by
Lemma~\ref{lemma:invariant-lift}.
\end{proof}

\begin{lemma}\label{lemma:d-7-q-3}
If $X$ is a surface and $\deg(X)=7$, then $-K_X\cdot H\ne 3$.
\end{lemma}

\begin{proof}
Suppose that $X$ is a surface such that $\deg(X)=7$ and $-K_X\cdot
H=3$. Then it follows from $(\ref{equation:RRR})$ and the
Riemann--Roch theorem that
$h^{0}(\mathcal{O}_{X}(K_X+H))\geqslant 3$. One the other hand,
there is an exact sequence of cohomology groups
$$
0\to H^{0}\Big(\mathcal{O}_{X}\big(K_{X}\big)\Big)\to H^{0}\Big(\mathcal{O}_{X}\big(K_{X}+H\big)\Big)\to H^{0}\Big(\mathcal{O}_{H}\big(K_{H}\big)\Big)\cong\C^3,%
$$
which implies that $h^{0}(\mathcal{O}_{X}(K_{X}+H))\leqslant 3$,
since $h^{0}(\mathcal{O}_{X}(K_{X}))=0$, because $X$ is of Fano
type. Thus, we see that $h^{0}(\mathcal{O}_{X}(K_X+H))=3$.

Suppose that there is a $\bar{G}$-invariant curve contained in the
base locus of the linear system $|K_X+H|$. Since $(K_X+H)\cdot
H=4$ and $\dim |K_X+H|>1$, this curve must be a disjoint union of
three lines, because $G$ is transitive. Let us denote these lines
by $L_1$, $L_2$, and $L_3$. Then $\bar{G}$ acts transitively on
these lines. One has $K_X+H\sim L_1+L_2+L_3+L$, where $L$ is a
line on $X$ such that $|L|$ is free from fixed components. In
particular, the linear system $|L|$ has at most one base point
since $L$ is a line, which implies that $|L|$ is base point free
by the transitivity of~$\bar{G}$, because the base locus of the
linear system~$|L|$ must be $\bar{G}$-invariant. Hence, we may
assume that $L$ is contained in the smooth locus of the surface
$X$, so that adjunction formula implies $-2=\big(K_X+L\big)\cdot
L\geqslant -H\cdot L=-1$, which is a contradiction. Therefore, the
base locus of the linear system $|K_X+H|$ consists of at most
finitely many points.

Suppose that $|K_X+H|$ is not composed of a pencil. Let $C_{1}$
and $C_{2}$ be two general curves in $|K_X+H|$, let $\Lambda_1$
and $\Lambda_2$ be the linear spans of the curves $C_1$ and $C_2$
in $\P^5$. Then the curves~$C_1$ and~$C_{2}$ are irreducible by
Bertini theorem, which implies that~$\Lambda_1$ and~$\Lambda_2$
are proper linear subspaces of $\P^5$. Note that
$\Lambda_1\neq\Lambda_2$, since $G$ is transitive. Hence, the set
$\Lambda_1\cap C_2$ consists of at most $\deg(C_2)=4$ points.
Moreover, the base locus of the linear system $|K_X+H|$ is
contained in $\Lambda_1\cap C_2$, which immediately implies that
$|K_X+H|$ is base point free, since  $G$ is transitive.

If $|K_X+H|$ is composed of some pencil $\mathcal{P}$, then
arguing as in the case when $|K_X+H|$ is not composed of a
pencil, we easily see the base locus of the pencil $\mathcal{P}$
consists of at most $4$ points, which implies that $|K_X+H|$ is
base point free, since the group $G$ is transitive and
$\mathcal{P}$ must be $\bar{G}$-invariant. Thus, we can conclude
that the linear system $|K_X+H|$ is base point free. In
particular, the divisor $K_X$ is Cartier, which implies that  $X$
has at most Du Val singularities, since we know already that $X$
has at most quotient singularities.

Note that $(7K_X+3H)\cdot H=0$. By Hodge index theorem one has
$(7K_X+3H)^2\leqslant 0$, which implies that $K_X^2\leqslant 1$.
Now we are going to show that $h^0(\O_X(-K_X))=0$.

Suppose that $h^0(\O_X(-K_X))\geqslant 2$. Since $G$ is
transitive, $|-K_X|$ is $\bar{G}$-invariant and \mbox{$-K_X\cdot
H=3$}, we see that the linear system $|-K_X|$ does not have fixed
curves. In particular, the divisor~$-K_X$ is nef. Recall that
$-K_X$ is big, since~$X$ is of Fano type. Therefore, the surface
$X$ must be a weak del Pezzo surface. In particular, we have
$K_{X}^2\geqslant 1$, which implies that $K_X^2=1$, because we
already proved that $K_X^2\leqslant 1$. It is well known that
$|-K_X|$ has a unique base point, which is impossible, since $G$
is transitive. Therefore, we see that $h^0(\O_X(-K_X))\leqslant
1$.

Suppose that $h^0(\O_X(-K_X))=1$. Since $-K_X\cdot H=3$ and~$G$ is
transitive, the only possibility is that the unique effective
divisor $D$ in $|-K_X|$ must be a disjoint union of three lines,
because~$G$ is transitive. Since $-K_X$ is a Cartier divisor and
$D$ is a smooth curve, we see that $D$ is contained in the smooth
locus of the surface $X$. This immediately leads to a
contradiction by adjunction formula. Therefore, we see that
$h^0(\O_X(-K_X))\ne 1$.

We see that $h^0(\O_X(-K_X))=0$. Note that
$h^2(\O_X(-K_X))=h^0(\O_X(2K_X))$ by the Serre duality, and
$h^0(\O_X(2K_X))=0$ since $H\cdot K_X<0$. Thus, the Riemann--Roch
theorem implies that $h^0(\O_X(-K_X))\geqslant 1+K_X^2$, so that
$K_X^2\le -1$. On the other hand, since the linear system
$|K_X+H|$ is base point free, one has $0\leqslant
(K_X+H)^2=K_X^2+1$, which implies that $K_X^2=-1$ and
$(K_X+H)^2=0$. Then the linear system $|K+H|$ gives us a morphism
$\psi\colon X\to\P^2$ such that $\psi(X)$ is a curve. So the
linear system $|K+H|$ is composed of a pencil. Note that
$\psi(X)$ can not be a line. In fact, one can easily see that
$\psi(X)$ must be a smooth conic, because $(K_X+H)\cdot H=4$.
Indeed, the degree of the curve $\psi(X)$ is a number of
irreducible components of a general element in  $|K+H|$. Since
$(K_X+H)\cdot H=4$, we see that~$\psi(X)$ is either an irreducible
conic or an irreducible quartic, since all irreducible components
of a general element in $|K+H|$ must have the same degree in
$\P^5$. If $\psi(X)$ is an irreducible quartic, then general curve
in $|K_{X}+H|$ is a disjoint union of four lines in $\P^5$, which
immediately leads to a contradiction with the adjunction formula.
So, we see that $\psi(X)$ must be a smooth conic.

Let $C$ be a fiber of the morphism $\psi$ over a general point in
$\psi(X)$. Then $C$ is an irreducible conic in $\P^2$, and
$K_{X}+H\sim 2C$. The pencil $|C|$ is base point free and it
induces a morphism $\vartheta\colon X\to\P^1$, whose general fiber
is an irreducible conic in $\P^5$. Thus, we obtain the diagram
$$
\xymatrix{
&X\ar[ld]_{\vartheta}\ar[rd]^{\psi}&&\\
\P^1\ar[rr]_{v_2}^{{}\cong{}}&&\psi(X)\ar@{^{(}->}[r]&\P^2, }
$$
where $v_2\colon\P^1\to\P^2$ is the second Veronese embedding.

Let us prove that  $X$ is smooth. Suppose that $X$ is singular. By
Theorem~\ref{theorem:Noether}, we have
\begin{equation}\label{eq:Noether-on-7-3}
\rk\Pic(X)+K_X^2+\sum_{P\in X}\mu(P)=10,
\end{equation}
where $\mu(P)$ is the Milnor number of a point $P\in X$. Since $G$
is transitive, it follows from $(\ref{eq:Noether-on-7-3})$ that
$|\mathrm{Sing}(X)|\geqslant 6$ and all singular points of the
surface $X$ are ordinary double points. Let~$r$ be the number of
reducible fibers of the morphism $\vartheta$. Since $H\cdot C=2$,
every reducible fiber of the morphism $\vartheta$ consists of two
different lines in $\P^5$ intersecting transversally in one point.
In particular, there is a $\bar{G}$-invariant subset in $X$ that
consists of exactly $r$ points, which implies that either $r=0$ or
$r\geqslant 6$. Note that $\rk\Pic(X)=r+2$. Therefore $r=0$, since
otherwise it would follow from $(\ref{eq:Noether-on-7-3})$ that
$10=r+|\mathrm{Sing}(X)|+1\geqslant 13$, which is contradiction.
Since $r=0$, it follows from $(\ref{eq:Noether-on-7-3})$ that
$|\mathrm{Sing}(X)|=9$. Since every fiber of the morphism
$\vartheta$ is irreducible rational curve and $X$ has only
ordinary double points, every fiber of $\vartheta$ contains either
none or two singular points of the surface $X$. Thus, nine
singular points of the surface $X$ must split in pairs, which is
impossible. Therefore, the surface $X$ is smooth.

Note that $h^2(\O_X(C-2K_X))=h^0(\O_X(3K_X-C))$ by the Serre
duality. Thus the Riemann--Roch theorem implies that
$$
h^0\Big(\O_X\big(C-2K_X\big)\Big)\geqslant 1+\frac{(C-2K_X)\cdot(C-3K_X)}{2}=3,%
$$
because $h^0(\O_X(3K_X-C))=0$, since $H\cdot (3K_X-C)=-11<0$. Now
we are going to prove that the linear system $|C-2K_X|$ has no
base curves.

Suppose that the linear system $|C-2K_X|$ contains a fixed curve.
Then there exists a non-zero effective divisor $Z$ on the surface
$X$ such that the linear system $|C-2K_X-Z|$ has no base curves,
i.\,e. $Z$ is a union of one dimensional components of the base
locus of the linear system $|C-2K_X|$ taken with with appropriate
multiplicities. Then $Z$ must be $\bar{G}$-invariant, which
implies that $H\cdot Z\geqslant 3$, since the group $G$ is
transitive. On the other hand, we have $H\cdot Z<H\cdot
(C-2K_X)=8$, since $h^0(\O_X(C-2K_X))>1$. Let $M$ be a general
curve in the linear system $|C-2K_X-Z|$. Then $1\leqslant M\cdot
H\leqslant 5$. Note that $K_X\cdot M\leqslant -1$, because $|M|$
is free from base curves and $-K_X$ is big, since $X$ is of Fano
type. Thus, we have
\begin{equation}
\label{eq:2CM} 0\leqslant 2C\cdot M=\big(K_X+H\big)\cdot M\leqslant H\cdot M-1.%
\end{equation}

Suppose that $H\cdot M\leqslant 2$. Then $C\cdot M=0$
by~(\ref{eq:2CM}), which implies that $M$ is contained in the
fibers of the morphism $\vartheta$. If $H\cdot M=1$, this means
that $M$ is a component of a reducible fiber of the morphism
$\vartheta$, which is impossible since $M^2\geqslant 0$. If
$H\cdot M=2$, then $M$ cannot consist of two components of
different reducible fibers of the morphism $\vartheta$, because
$M^2\geqslant 0$. Thus, if $H\cdot M=2$, then $M\in |C|$, since
$H\cdot M=H\cdot C$, which is impossible, since
$$h^0(\O_X(C-2K_X))\geqslant 3>2=h^0(\O_X(C)).$$
Therefore, we see that $H\cdot M\geqslant 3$.

Suppose that $H\cdot M=3$. Arguing as above, we see that $M$ is
not contained in the fibers of the morphism $\vartheta$. Thus, we
must have $C\cdot M\geqslant 1$, so that $C\cdot M=1$
by~(\ref{eq:2CM}). In particular, the linear system $|M|$ is not
composed of a pencil unless $|M|$ is a pencil itself, since
otherwise one would have $C\cdot M\geqslant 2$. In particular, the
curve $M$ is irreducible by the Bertini theorem. Therefore, the
curve $M$ is a section of the morphism $\vartheta$, which implies
that $M$ is smooth and rational. Applying the adjunction formula
to the curve $M$, one obtains
$$-2=(K_X+M)\cdot M=-1+M^2\geqslant -1,$$
which is a contradiction. Thus, we see that $H\cdot M\ne 3$.

Suppose that $H\cdot M=4$ and $C\cdot M\geqslant 1$. Then $C\cdot
M=1$ by~(\ref{eq:2CM}). In particular, the linear system $|M|$ is
not composed of a pencil  unless $|M|$ is a pencil itself, since
otherwise one would have $C\cdot M\geqslant 2$. Hence, the curve
$M$ is irreducible by the Bertini theorem. Therefore,
the curve~$M$ is a section of the morphism $\vartheta$, which implies that
$M$ is smooth and rational. Applying the adjunction formula to the
curve $M$, wee see that $-2=(K_X+M)\cdot M=-2+M^2$, which implies
that $M^2=0$. Hence, the linear system $|M|$ is composed of a
pencil, which implies that~$|M|$ is a pencil, since $M$ is
irreducible. The latter is impossible since we already proved that
$h^0(\O_X(M))\geqslant 3$. Thus, we see that if $H\cdot M=4$, then
$C\cdot M=0$.

Suppose that $H\cdot M=4$ and $C\cdot M=0$. Since $|M|$ is free
from base curves, we must have $M\sim 2C$, because $H\cdot
M=2H\cdot C$. Note that $H\cdot Z=4$. Thus, since the group $G$ is
transitive, the divisor $Z$ is either a union of two different
conics in $\P^5$ or a union of four different lines in~$\P^5$. Let
us consider these cases separately.

Suppose that $Z$ is a union of two different conics in $\P^5$. Let
us denote these conics by~$R_{1}$ and~$R_{2}$. Let $\Lambda_{1}$
and $\Lambda_{2}$ be linear spans in $\P^5$ of the conics $R_{1}$
and $R_{2}$. Then $\Lambda_{1}\cap \Lambda_{2}=\varnothing$,
because~$G$ is transitive and  $Z$ is $\bar{G}$-invariant. On the
other hand, we have $C\cdot Z=4$, which implies that both
intersections $C\cap R_{1}$ and $C\cap R_{2}$ consist of two
points, because the conic~$C$ is a fiber of the morphism $\psi$
over a general point in $\psi(X)$. Thus, the linear span in $\P^5$
of the conic~$C$ intersects both linear subspaces $\Lambda_1$ and
$\Lambda_2$ by lines, which is impossible since~$\Lambda_1$
and~$\Lambda_2$ are disjoint. Thus, we see that $Z$ can not be a
union of two different conics in $\P^5$.

We see that $Z$ is  a union of four different lines. Let us denote
these lines by $L_{1}$, $L_{2}$, $L_{3}$, and $L_{4}$. Then
$\bar{G}$ acts transitively on these lines, because $G$ is
transitive. It follows from the Riemann--Roch theorem that
$$
h^0\Big(\O_X\big(H-C\big)\Big)\geqslant 1+\frac{(H-C)\cdot (H-C-K_X)}{2}=3,%
$$
since $h^2(\O_X(H-C))=h^0(\O_X(K_X-H+C))$ by the Serre duality,
and $h^0(\O_X(K_X-H+C))=0$, because $H\cdot (K_X-H+C)=-8<0$. By
Theorem~\ref{theorem:Noether}, we have $\rho(X)=11$, which implies
that the morphism $\vartheta$ has $9$ reducible fibers. Thus, it
follows from Lemma~\ref{lemma:small-orbit-on-P1} that there exist
two curves $C_1$ and $C_2$ in the pencil $|C|$ such that the
divisor $C_1+C_2$ is $\bar{G}$-invariant. Since $G$ is transitive,
we see that $C_1\cup C_2$ is not contained in a hyperplane in
$\P^5$. In particular, one has $h^0(\O_X(H-2C))=0$. On the other
hand,  there is an exact sequence of cohomology groups
$$
0\to H^0\Big(\O_X\big(H-2C\big)\Big)\to H^0\Big(\O_X\big(H-C\big)\Big)\to H^0\Big(\O_C\big(H-C\big)\Big),%
$$
which implies that $h^0(\O_X(H-C))=3$, since $h^0(\O_X(H-2C))=0$,
$h^0(\O_C(H-C))=3$ and $h^0(\O_X(H-C))\geqslant 3$. On the other
hand, one can easily see that the linear system $|H-C|$ is base
point free, since the linear spans in $\P^5$ of the conics $C_{1}$
and $C_{2}$ are disjoint. Furthermore, the linear system $|H-C|$
is not composed of a pencil since $(H-C)^2=3>0$. Let $\tau\colon
X\to\P^2$ be a morphism that is given by the linear system
$|H-C|$. Then $\tau$ is $\bar{G}$-equivariant, surjective and
generically three-to-one. Note that $\tau$ can be considered as a
projection from a two dimensional linear subspace in $\P^5$. Since
$L_{i}\cdot(H-C)=0$ for every $i\in\{1,2,3,4\}$, the lines
$L_{1}$, $L_{2}$,~$L_{3}$,~and~$L_{4}$ are contracted by $\tau$.
Moreover, the points $\tau(L_{1})$, $\tau(L_{2})$, $\tau(L_{3})$,
and $\tau(L_{4})$ are different and are not contained in a line in
$\P^2$, because $L_{1}\cup L_{2}\cup L_{3}\cup L_{4}$ is not
contained in a hyperplane in $\P^5$, since $G$ is transitive. Let
$\iota\colon\bar{G}\to\Aut(\P^2)$ be the homomorphism induced by
the morphism $\tau$. Then $\iota(\bar{G})$ acts faithfully on the
set $\{\tau(L_{1}), \tau(L_{2}), \tau(L_{3}), \tau(L_{4})\}$ since
every element in $\Aut(\P^2)$ is defined by the images of~$4$
points in general position. In particular, we see that
$\iota(\bar{G})$ is a subgroup of the group $\SS_4$, which implies
that there is a $\iota(\bar{G})$-invariant conic in $\P^2$.
Therefore, there exists a $\bar{G}$-invariant divisor in the
linear system $|2H-2C|$. Let us denote this divisor by~$B$. Then
$B+C_{1}+C_{2}\sim 2H$, and $B+C_{1}+C_{2}$ is
$\bar{G}$-invariant, which is impossible by
Lemma~\ref{lemma:invariant-lift}. The obtained contradiction shows
that $H\cdot M\ne 4$.

Suppose that $H\cdot M=5$. Then $H\cdot Z=3$, which implies that
the curve $Z$ is a union of three lines transitively interchanged
by~$\bar{G}$, because $G$ is transitive. Let us denote these lines
by $L_{1}$, $L_{2}$ and $L_{3}$. We have $3L_j\cdot C=4-C\cdot M$
for every $j\in\{1,2,3\}$. Since $0\leqslant C\cdot M\leqslant 4$
by~(\ref{eq:2CM}), we see that $C\cdot M=1$. In particular, the
linear system $|M|$ is not composed of a pencil unless $|M|$ is
a pencil itself, since otherwise one would have $C\cdot M\ge 2$.
Hence, the curve $M$ is irreducible by the Bertini theorem.
Therefore, the curve~$M$ is a section of the morphism $\vartheta$,
which implies that~$M$ is smooth and rational. Furthermore, we
have $K_{X}\cdot M=(2C-H)\cdot M=-3$. Now applying the adjunction
formula to~$M$, we see that $-2=(K_X+M)\cdot M=-3+M^2$, which
implies that $M^2=1$. Thus, it follows from the Riemann--Roch
theorem that
$$
h^0(\O_X(M-C))\ge 1+\frac{(M-C)\cdot (M-C-K_X)}{2}=1,
$$
because $h^2(\O_X(M-C))=h^0(\O_X(K_X-M+C))=0$ by the Serre
duality, since one has $H\cdot (K_X-M+C)=-6<0$. Keeping in mind
that $(M-C)^2=-1$, we see that the base locus of the linear system
$|M-C|$ contains a $\bar{G}$-invariant curve. Let us denote this
curve by $R$. Then $H\cdot R\leqslant H\cdot(M-C)=3$, which
implies that the curve $R$ is a union of three lines transitively
interchanged by~$\bar{G}$, since $G$ is transitive. Let us denote
these lines by $L_{1}^{\prime}$, $L_{2}^{\prime}$ and
$L_{3}^{\prime}$. Then
$$
1=C\cdot M=C\cdot\big(M-C\big)=C\cdot\Big(L_{1}^{\prime}+L_{2}^{\prime}+L_{3}^{\prime}\Big)=3C\cdot L_{1}^{\prime}=3C\cdot L_{2}^{\prime}=3C\cdot L_{2}^{\prime},%
$$
which is a contradiction. Therefore
$H\cdot M\ne 5$.

Thus, the assumption that $|C-2K_{X}|$ has a fixed curve implies
that $H\cdot M>5$. However, we proved earlier that $1\leqslant
M\cdot H\leqslant 5$. So, we see that $|C-2K_{X}|$ has no based
components, i.e.  the curve $Z$ does not exists. On the other
hand, one has $-K_X\cdot(C-2K_X)=0$, which is impossible, because
the divisor $-K_X$ is big, since the surface $X$ is of Fano type.
The obtained contradiction completes the proof of
Lemma~\ref{lemma:d-7-q-3}.
\end{proof}

Now using Lemmas~\ref{lemma:d-4-q-6}, \ref{lemma:d-5-q-5},
\ref{lemma:d-6-q-4}, and \ref{lemma:d-7-q-3}, we are ready to
prove the following

\begin{lemma}\label{lemma:dim-2}
The subvariety $X$ is not a surface.
\end{lemma}

\begin{proof}
Suppose that $X$ is a surface. Then it follows from
the~Riemann--Roch theorem that
\begin{equation}
\label{equation:dim-6-RR-2}
h^{0}\Big(\mathcal{O}_{X}\big(nH\big)\Big)=\chi\Big(\mathcal{O}_{X}\big(nH\big)\Big)=1+\frac{n^{2}}{2}\Big(H\cdot H\Big)-\frac{n}{2}\Big(H\cdot K_{X}\Big)%
\end{equation}
for any $n\geqslant 1$.  In particular, since $X$ is not contained
in a~hyperplane in $\mathbb{P}^{5}$, it follows from
$(\ref{equation:exact-sequence-n-6})$ that  $H\cdot H-H\cdot
K_{X}=10$. On the other hand, we know that $H\cdot H\geqslant 4$,
since $S$ is not contained in a~hyperplane in $\mathbb{P}^{5}$.
Moreover, since $X$ is of Fano type, the divisor $-K_{X}$ is big,
which implies that $-H\cdot K_{X}\geqslant 1$. Thus, we see that
$H\cdot H\in\{4,5,6,7,8,9\}$. Now plugging $n=2$ into
$(\ref{equation:exact-sequence-n-6})$ and
$(\ref{equation:dim-6-RR-2})$, we see that
$h^{0}(\mathcal{O}_{\mathbb{P}^{5}}(2)\otimes\mathcal{I}_{X})=10-H\cdot
H$. Then $H\cdot H\leqslant 7$ by
Lemma~\ref{lemma:pencil-of-quadrics}, which implies that $H\cdot
H\in\{4,5,6,7\}$. Since $H\cdot H-H\cdot K_{X}=10$, we immediately
obtain a contradiction using Lemmas~\ref{lemma:d-4-q-6},
\ref{lemma:d-5-q-5}, \ref{lemma:d-6-q-4}, and \ref{lemma:d-7-q-3}.
\end{proof}

Thus, we see that $X$ is a threefold. Put $d=H\cdot H\cdot H$.

\begin{lemma}\label{lemma:dim-3-deg-4}
The inequality $d\geqslant 5$ holds.
\end{lemma}

\begin{proof}
Since $G$ is transitive, the threefold $X$ is not contained in a
hyperplane. Furthermore, the threefold $X$ is not a cone since the
vertex of a cone is a linear space, which again contradicts
transitivity of $\bar{G}$. Hence it follows
from~\cite[Theorem~1]{SwinnertonDyer73} that $X$ is either a
complete intersection of two quadrics, or a projection from $\P^6$
of a hyperplane section of the Segre variety
\mbox{$\P^1\times\P^3\subset\P^7$}. The former case is impossible
by assumption, and the latter case is impossible since $X$ is
projectively normal.
\end{proof}

By Lemma~\ref{lemma:dim-3-deg-4}
the threefold $X$ is not contained in an intersection
of two quadrics in $\P^5$ (indeed, otherwise either $X$ would
be contained in a hyperplane, or one would have $d\leqslant 4$).
On the other hand,
$h^{0}(\mathcal{O}_{\mathbb{P}^{5}}(2)\otimes\mathcal{I}_{X})\neq 1$
since~$\bar{G}$ does not have invariant quadrics in~$\P^5$.
Therefore, we have the following

\begin{corollary}\label{corollary:dim-3-deg-4}
The equality
$h^{0}(\mathcal{O}_{\mathbb{P}^{5}}(2)\otimes\mathcal{I}_{X})=0$
holds.
\end{corollary}

Recall that $f\colon\tilde{X}\to X$ is a resolution of
singularities of the threefold $X$. Let
$\mathrm{c}_{2}(\tilde{X})$ be the~second Chern class of
the~threefold $\tilde{X}$. Put
$\gamma=f^{*}(H)\cdot(K_{\tilde{X}}\cdot
K_{\tilde{X}}+\mathrm{c}_{2}(\tilde{X}))$ and $k=-H\cdot H\cdot
K_{X}$. Then $k\geqslant 1$,  since $X$ is of Fano type. Then
$$
h^{0}\Big(\mathcal{O}_{X}\big(nH\big)\Big)=\chi\Big(\mathcal{O}_{X}\big(nH\big)\Big)=d\frac{n^{3}}{6}+k\frac{n^{2}}{4}+\frac{n}{12}\gamma+1%
$$
by the Riemann--Roch theorem. Using
$(\ref{equation:exact-sequence-n-6})$ and
Corollary~\ref{corollary:dim-3-deg-4}, we see~that
\begin{equation}
\label{equation:RRRR}
\left\{\aligned
&\frac{d}{6}+\frac{k}{4}+\frac{\gamma}{12}+1=6,\\
&\frac{4d}{3}+k+\frac{\gamma}{6}+1=21,\\
&\frac{9d}{2}+\frac{9k}{4}+\frac{\gamma}{4}+1=56-h^{0}(\mathcal{O}_{\mathbb{P}^{5}}(3)\otimes\mathcal{I}_{X}),\\
\endaligned
\right.
\end{equation}
which implies that $10=d+k/2$ and
$h^{0}(\mathcal{O}_{\mathbb{P}^{5}}(3)\otimes\mathcal{I}_{X})=k/2$.
Keeping in mind that $d\geqslant 5$ by Lemma~\ref{lemma:dim-3-deg-4}
and $k\ge 1$,
we see that $d\in\{5,6,7,8,9\}$.

Note that $d\neq 9$, since otherwise one would have
$h^{0}(\mathcal{O}_{\mathbb{P}^{5}}(3)\otimes\mathcal{I}_{X})=1$
so that~$\bar{G}$ would have an invariant cubic in~$\P^5$
which is not the case by assumption.
Finally, if $d=8$, then
$h^{0}(\mathcal{O}_{\mathbb{P}^{5}}(3)\otimes\mathcal{I}_{X})=2$
so that there is a~$\bar{G}$-invariant pencil~$\mathcal{P}$
of cubics in~$\mathbb{P}^5$. For two different
hypersurfaces $R_1, R_2\in\mathcal{P}$ the intersection
$R_1\cap R_2$ consists of the threefold~$X$ and a threefold~$T$
of degree $\deg(T)=9-d=1$. Thus~$T$ is a $\bar{G}$-invariant
projective subspace of~$\mathbb{P}^5$,
which is impossible by assumption.

The next two lemmas deal with the (slightly more
difficult) cases $d=5$ and $d=7$.

\begin{lemma}\label{lemma:dim-3-deg-5}
The inequality $d\ne 5$ holds.
\end{lemma}
\begin{proof}
Suppose that $d=5$. Then $k=10$, because $10=d+k/2$. Since
$(K_X+2H)\cdot H^2=0$, the sectional genus of $X$ is $1$. Note
that $X$ is rationally connected by \cite{Zh06}, because $X$ is of
Fano type. Applying~\cite[Theorem~5.4]{Fujita1989} and keeping in
mind that $X$ can be neither a scroll over an elliptic curve nor a
generalized cone over such scroll, we get
$$1=3+\deg(X)-h^0(\O_X(H))=2,$$
which is a contradiction.
\end{proof}

\begin{lemma}\label{lemma:dim-3-deg-7}
The inequality $d\ne 7$ holds.
\end{lemma}

\begin{proof}
Suppose that $d=7$. Then $k=6$ and
$h^{0}(\mathcal{O}_{\mathbb{P}^{5}}(3)\otimes\mathcal{I}_{X})=3$,
because $10=d+k/2$ and
$h^{0}(\mathcal{O}_{\mathbb{P}^{5}}(3)\otimes\mathcal{I}_{X})=k/2$.
Thus, an intersection of $X$ with a general linear subspace in
$\P^5$ of codimension $2$ is a smooth curve of genus $5$ and
degree $7$ (cf. Theorem~\ref{theorem:curve-degree-7-genus-5}).

Let $\mathcal{P}$ be the linear system of cubic hypersurfaces in
$\P^5$ that contains $X$. Then $X$ is the only threefold in the
base locus of the linear system $\mathcal{P}$, because $G$ is
transitive. Let $Y_1, Y_2, Y_3$ be three general hypersurfaces in
$\mathcal{P}$. Then $Y_1\cap Y_2=X\cup Q$, where $Q$ is a
threefold of degree $\deg(Q)=2$. Moreover, it follows from
Corollary~\ref{corollary:curve-degree-7-genus-5-special} that $Q$
is a (reduced) irreducible quadric threefold.

Let $H(Y_1, Y_2)$ be the unique hyperplane in $\P^5$ that contains
$Q$. Then
$$
\bigcap_{Y_1, Y_2}H(Y_1, Y_2)=\varnothing,
$$
since otherwise this intersection would be a $\bar{G}$-invariant
proper linear subspace of $\P^5$, which is impossible, because $G$
is transitive.

Put $S=Q\cap Y_3$. Then $S$ is a reduced surface of degree $6$ by
Corollary~\ref{corollary:curve-degree-7-genus-5-special}. On the
other hand, the intersection $H(Y_1, Y_2)\cap X$ is a surface that
contains the surface $S$. Therefore, the scheme-theoretic
intersection $H(Y_1, Y_2)\cdot X$ is reduced and consists of a
union of a surface $S$ and some two-dimensional linear subspace
$\Pi\subset\P^5$. (Note that the existence of the plane~$\Pi$
follows form \cite[Proposition~8]{Okonek84} in the case when $X$
is smooth.)

A priori, the plane $\Pi$ depends on the choice of the divisors
$Y_1$, $Y_2$, and $Y_3$ in the linear system~$\mathcal{P}$. In
fact, the plane $\Pi$ must vary when one varies  the divisors
$Y_1$, $Y_2$, and $Y_3$ in the linear system~$\mathcal{P}$,
because $G$ is transitive. Hence, the surface $H$ is swept out by
lines, which implies that its the Kodaira dimension $\kappa(H)$ is
$-\infty$.

Let $R$ be the general hyperplane section of $H\subset
X\subset\P^5$. Then $R\cdot K_H=d-k$, which implies that $H\cdot
K_H=1$. Thus, we have $\chi(\O_H(R))=\chi(\mathcal{O}_{H})+3$ by
the Riemann--Roch theorem. Since $H$ is projectively normal, we
have $h^0(\O_H(R))=5$. On the other hand, we know that
$h^1(\O_H(R))=h^2(\O_H(R))=0$. Since
$\chi(\O_H(R))=\chi(\mathcal{O}_{H})+3$, one has $\chi(\O_H)=2$,
which implies that $h^{2}(\mathcal{O}_{H})\geqslant 1$. Let
$\iota\colon\bar{H}\to H$ be a resolution of singularities. Then
$h^{2}(\mathcal{O}_{\bar{H}})=h^{2}(\mathcal{O}_{H})$, because $H$
has rational singularitis. But
$h^{0}(\mathcal{O}_{\bar{H}}(K_{X}))=h^{2}(\mathcal{O}_{\bar{H}})$
by the Serre duality, which implies that $\bar{H}$ is a surface of
general type, which is impossible since
$\kappa(H)=\kappa(\bar{H})=-\infty$.
\end{proof}

Therefore, we see that $d=6$. Then $k=8$ and
$h^{0}(\mathcal{O}_{\mathbb{P}^{5}}(3)\otimes\mathcal{I}_{X})=4$,
because $10=d+k/2$ and
$h^{0}(\mathcal{O}_{\mathbb{P}^{5}}(3)\otimes\mathcal{I}_{X})=k/2$,
which implies that the sectional genus of the threefold $X$ is $3$
by the adjunction formula. Thus, we proved that $X$ is a
irreducible $\bar{G}$-invariant projectively normal non-degenerate
Fano type threefold of degree $6$ and sectional genus~$3$ such
that
$h^{0}(\mathcal{O}_{\mathbb{P}^{5}}(2)\otimes\mathcal{I}_{X})=0$
and
$h^{0}(\mathcal{O}_{\mathbb{P}^{5}}(3)\otimes\mathcal{I}_{X})=4$.
This completes the proof of Theorem~\ref{theorem:6-dim-1}.

\appendix

\section{Curves of genus $5$ and degree $7$}
\label{section:septic-curves}

Let $C$ be a smooth irreducible curve in $\mathbb{P}^{3}$ of genus
$5$ and degree $7$, let $\pi\colon X\to\mathbb{P}^{3}$ be a blow
up of the curve $C$, let $E$ be the $\pi$-exceptional divisor, and
let $H$ be a general hyperplane in $\mathbb{P}^{3}$.

\begin{theorem}
\label{theorem:curve-degree-7-genus-5} The divisor $-K_{X}$ is
ample, and $|\pi^{*}(3H)-E|$ is free from base points and induces
a morphism $\phi\colon X\to\mathbb{P}^{2}$ that is a conic bundle
with a discriminant curve of degree $5$.
\end{theorem}

\begin{proof}
The required assertion is well-known in the case when $C$ is a
scheme-theoretic intersection of cubic hypersurfaces in
$\mathbb{P}^{3}$ (see \cite[Proposition~7.5]{MoriMukai85}). It is
known that generic smooth connected curve in $\mathbb{P}^{3}$ of
genus $5$ and degree $7$ is a scheme-theoretic intersection of
cubic hypersurfaces in $\mathbb{P}^{3}$ (see
\cite[Corollary~6.2]{MoriMukai85},
\cite[Proposition~7.5]{MoriMukai85}). Unfortunately, we failed to
find any reference with a proof that the same holds for any smooth
connected curve in~$\mathbb{P}^{3}$ of genus~$5$ and degree $7$.
So we decided to prove it here.

Recall that the curve $C$ is called $m$-regular in the case when
$$
H^{1}\Big(\mathcal{O}_{\mathbb{P}^{3}}\big(m-1\big)\otimes\mathcal{I}_{C}\Big)=H^{2}\Big(\mathcal{O}_{\mathbb{P}^{3}}\big(m-2\big)\otimes\mathcal{I}_{C}\Big)=0,
$$
where $\mathcal{I}_{C}$ is the ideal sheaf of $C$. It is
well-known that $C$ is  a scheme-theoretic intersection of
hypersurfaces of degree $m$ in $\mathbb{P}^{3}$ if it is
$m$-regular (see \cite[Excersize~20.21]{Eisenbud}). However, our
curve $C$ is not $3$-regular, since
$h^{2}(\mathcal{O}_{\mathbb{P}^{3}}(1)\otimes\mathcal{I}_{C})=1$.
On the other hand, it follows from
\cite[Exercise~D.14(6)]{ArbarelloCornalba} that the curve $C$ is
projectively normal, which implies that
$h^{1}(\mathcal{O}_{\mathbb{P}^{3}}(n)\otimes\mathcal{I}_{C})=0$
for every non-negative integer $n$. Moreover, it follows from the
Riemann--Roch theorem that
$h^{2}(\mathcal{O}_{\mathbb{P}^{3}}(2)\otimes\mathcal{I}_{C})=0$,
which implies that $C$ is $4$-regular. Thus, the curve $C$ is a
scheme-theoretic intersection of quartic hypersurfaces in
$\mathbb{P}^{3}$. In particular, the divisor $-K_{X}$ is nef.
Since $-K_{X}^{3}=16$ (this is obvious), we see that $-K_{X}$ is
big and nef.

It follows from the projective normality of the curve $C$ that $C$
is not contained in any quadric hypersurface in $\P^3$ and
$h^{0}(\mathcal{O}_{\mathbb{P}^{3}}(3)\otimes\mathcal{I}_{C})=3$,
which  implies that $h^{0}(\mathcal{O}_{X}(\pi^{*}(3H)-E))=3$.

Let $P$ be a sufficiently general point in $\P^3$ that is not
contained in the image under the morphism~$\pi$ of the base locus
of the linear system $|\pi^{*}(3H)-E|$, and let $S_{1}$ and
$S_{2}$ be two general cubic surfaces in $\P^3$ that pass through
$C$ and contain $P$. Then $S_{1}\cdot S_{2}=C+\Delta$, where
$\Delta$ is an effective one-cycle in $\P^3$ such that
$P\in\mathrm{Supp}(\Delta)$. Since $H\cdot\Delta=2$, we see that
$C\not\subset\mathrm{Supp}(\Delta)$, which implies that $S_{1}$
and $S_{2}$ are smooth in the general point of the curve $C$, the
base locus of the linear system $|\pi^{*}(3H)-E|$ does not contain
fixed components and does not contain curves in $E$ that are not
contracted by $\pi$ to points in $C$.

To complete the proof it is enough to prove that $|\pi^{*}(3H)-E|$
is base point free. Note that the linear system $|\pi^{*}(3H)-E|$
is base point free if its base locus does not contains curves that
are not contained in $E$, since $(\pi^{*}(3H)-E)^3=0$ and
$|\pi^{*}(3H)-E|$ does not contain curves in~$E$ that are not
contracted by $\pi$ to points in $C$. In particular, we see that
$|\pi^{*}(3H)-E|$ is base point free if $C$ is a set-theoretic
intersection of cubic hypersurfaces in $\P^3$.

Suppose that the base locus of the linear system $|\pi^{*}(3H)-E|$
contains an irreducible curve~$\tilde{Z}$ such that
$\tilde{Z}\not\subset E$. Let us show that this assumption leads
to a contradiction.

Let $\tilde{S}_{1}$ and $\tilde{S}_{2}$ be the proper transforms
of the surfaces $S_{1}$ and $S_{2}$ on the threefold $X$,
respectively. Then $\tilde{S}_{1}$ and $\tilde{S}_{2}$ are general
surfaces in $|\pi^{*}(3H)-E|$. Let $\tilde{\Delta}$ be the proper
transform of the one-cycle $\Delta$ on the threefold $X$. Then
$\tilde{S}_{1}\cdot\tilde{S}_{2}=\tilde{\Delta}+\tilde{F}$, where
$\tilde{F}$ is an effective one-cycle on $X$ whose support
consists of finitely many curves in $X$ that are contracted by
$\pi$ to points in $C$. By assumption, we know that
$\tilde{Z}\subset\tilde{\Delta}$. Put $Z=\pi(\tilde{Z})$. Then
$Z\ne\Delta$, since $P\not\in Z$ and $P\in\mathrm{Supp}(\Delta)$.
We see that $Z$ is a line and $\tilde{Z}$ the unique curve in the
base locus of the linear system  $|\pi^{*}(3H)-E|$ that is not
contained in the surface $E$, because $2=H\cdot\Delta>H\cdot
Z\geqslant 1$. In particular, we see that $\Delta$ is reduced.

Let us show that the divisor $\pi^{*}(3H)-E$ is nef. Suppose that
this is not the case. Then $(\pi^{*}(3H)-E)\cdot \tilde{Z}<0$. But
$0>(\pi^{*}(3H)-E)\cdot\tilde{Z}=3-E\cdot\tilde{Z}$ and
$4-E\cdot\tilde{Z}=-K_{X}\cdot\tilde{Z}\geqslant 0$, which implies
that $Z$ is a $4$-section of the curve $C$. Since it follows from
the Riemann--Roch theorem that
$h^{0}(\mathcal{O}_{C}(K_{C}-H\vert_C))=1$, the existence of a
$4$-section to the curve $C$ immediately implies that the curve
$C$ is either trigonal or hyperelliptic, which is not the case
(see the proof of \cite[Corollary~6.2]{MoriMukai85}).

Thus, we see that $\pi^{*}(3H)-E$ is nef. In particular, the
divisor $-K_{X}$ is ample. Then $\tilde{F}=0$, because
$(\pi^{*}(3H)-E)^3=0$ and $\pi^{*}(3H)-E$ is $\pi$-ample.
Similarly, we see the base locus of the linear system
$|\pi^{*}(3H)-E|$ consists of the curve $\tilde{Z}$. We have
$\tilde{\Delta}=\tilde{Z}+\tilde{R}$, where $\tilde{R}$ is an
irreducible curve on $X$ such that $\pi(\tilde{R})$ is a line.
Then
$$
0=\Big(\pi^{*}\big(3H\big)-E\Big)^3=\Big(\pi^{*}\big(3H\big)-E\Big)\cdot\tilde{S}_{1}\cdot\tilde{S}_{2}=\Big(\pi^{*}\big(3H\big)-E\Big)\cdot\tilde{Z}+\Big(\pi^{*}\big(3H\big)-E\Big)\cdot\tilde{R},%
$$
which implies that
$(\pi^{*}(3H)-E)\cdot\tilde{Z}=(\pi^{*}(3H)-E)\cdot\tilde{R}=0$,
because  $\pi^{*}(3H)-E$ is nef. In particular, the curves
$\tilde{R}$ and $\tilde{Z}$ both generate the extremal ray of the
cone of effective cycles $\overline{\mathrm{NE}}(X)$ that is
different from the ray contracted by $\pi$. Put
$R=\pi(\tilde{R})$. Then both lines $Z$ and $R$ must be
$3$-secants of the curve $C$, because
$(\pi^{*}(3H)-E)\cdot\tilde{Z}=(\pi^{*}(3H)-E)\cdot\tilde{R}=0$.

Since the base locus of the linear system $|\pi^{*}(3H)-E|$
consists of the curve $\tilde{Z}$, we must have
$\tilde{R}\cap\tilde{Z}=\varnothing$, because
$(\pi^{*}(3H)-E)\cdot\tilde{R}=0$. Then $R\cap Z=\varnothing$.
Indeed, suppose that $R\cap Z\ne\varnothing$. Then $R\cap Z$ is a
point contained in~$C$. Let $\tilde{L}$ be the curve in $X$ such
that $\pi(\tilde{L})=R\cap Z$. Then $\tilde{R}\cap\tilde{L}$ is a
point, and $\tilde{L}\subset\tilde{S}_{1}$, because otherwise we
must have
$$
1=\Big(\pi^{*}\big(3H\big)-E\Big)\cdot\tilde{L}\geqslant\mathrm{mult}_{\tilde{L}\cap\tilde{R}}\big(\tilde{S}_{1}\big)+\mathrm{mult}_{\tilde{L}\cap\tilde{Z}}\big(\tilde{S}_{1}\big)\geqslant 2,%
$$
which is absurd. Similarly, we see that
$\tilde{L}\subset\tilde{S}_{2}$, which is impossible, because we
already proved that $\tilde{F}=0$. Thus, we see that $R\cap
Z=\varnothing$, i.e. the lines $R$ and $Z$ are
disjoint\footnote{As was pointed out to us by E.\,Ballico, the
curve $\Delta$ is connected, because $C$ is arithmetically
Cohen--Macaulay and thus $\Delta$ is arithmetically
Cohen--Macaulay as well (see
\cite[Proposition~1.2]{PeskineSzpiro1974}). Then $\Delta$ is a
reducible conic.}.

Let $\gamma\colon X\to B$ be a contraction of the extremal ray in
$\overline{\mathrm{NE}}(X)$ that is generated by $\tilde{R}$.
Then~$\gamma$ must be a conic bundle and $B$ must be a smooth
surface (see \cite[Proposition~4.16]{MoriMukai85}).
Since~$\pi(\tilde{R})$ contains the point $P$, which is a
sufficiently general point in $\P^3$, we see that $\gamma$ must be
a $\mathbb{P}^{1}$ bundle, which is impossible by
\cite{SzurekWisniewski90}. The obtained contradiction completes
the proof.
\end{proof}

\begin{corollary}[cf. {\cite[Theorem~3.2]{Homma80}}]
\label{corollary:curve-degree-7-genus-5} The curve $C$ is  a
scheme-theoretic intersection of cubic hypersurfaces in
$\mathbb{P}^{3}$.
\end{corollary}

\begin{corollary}
\label{corollary:curve-degree-7-genus-5-special} Let $S_{1}$,
$S_{2}$ and $S_{3}$ be three general cubic surfaces in $\P^3$ that
pass through $C$. Then $S_{1}\cdot S_{2}=C+\Delta$ for an
irreducible reduced conic $\Delta\subset\P^3$ such that the
zero-cycle $S_{3}\cdot \Delta$ is reduced and consists of $7$
distinct point in $\Delta$.
\end{corollary}

\end{document}